\documentclass[12pt,twoside]{article}
\usepackage{amsmath,amsfonts,amssymb,a4,enumitem,epsfig,array,psfrag}
\usepackage{url}
\usepackage{amsthm}
\usepackage{etoolbox} 
\usepackage{hyperref}

\usepackage{color}

\newtheorem{theo}{Theorem}
\newtheorem{prop}{Proposition}[section]
\newtheorem{fact}[prop]{Fact}
\newtheorem{coro}[prop]{Corollary}
\newtheorem{lemma}[prop]{Lemma}

\theoremstyle{definition}
\newtheorem{defin}[prop]{Definition}
\newtheorem{example}[prop]{Example}
\AtEndEnvironment{example}{\null\hfill$\diamond$}%
\newtheorem{remark}[prop]{Remark}
\AtEndEnvironment{remark}{\null\hfill$\diamond$}%

\newcommand{\bff}{{\mathbf f}}
\newcommand{\bt}{{\mathbf t}}

\newcommand{\emptyplug}{{\mathbf p_\circ}}
\newcommand{\fullplug}{{\mathbf p_\bullet}}
\newcommand{\plug}{\operatorname{plug}}
\newcommand{\floorop}{\operatorname{floor}}

\newcommand{\nvert}{\operatorname{vert}}

\newcommand{\thin}{\operatorname{thin}}
\newcommand{\flux}{\operatorname{flux}}
\newcommand{\Flux}{\operatorname{Flux}}

\newcommand{\black}{\operatorname{bl}}
\newcommand{\white}{\operatorname{wh}}
\newcommand{\hz}{\operatorname{hz}}
\newcommand{\vt}{\operatorname{vt}}

\newcommand{\sign}{\operatorname{sign}}
\newcommand{\tk}{\operatorname{tk}}
\newcommand{\inv}{\operatorname{inv}}
\newcommand{\Inv}{\operatorname{Inv}}

\newcommand{\NN}{{\mathbb{N}}}
\newcommand{\ZZ}{{\mathbb{Z}}}

\newcommand{\RR}{{\mathbb{R}}}

\newcommand{\cT}{{\cal T}}

\newcommand{\cC}{\mathcal{C}}
\newcommand{\cD}{{\cal D}}

\newcommand{\cP}{{\cal P}}
\newcommand{\cR}{{\cal R}}
\newcommand{\cG}{{\cal G}}

\newcommand{\Tw}{\operatorname{Tw}}

\newcommand{\tw}{\operatorname{tw}}

\begin{document}
\title{Domino tilings and flips \\ in dimensions $4$ and higher}
\author{Caroline Klivans \and Nicolau C. Saldanha}

\maketitle

\begin{abstract}

In this paper we consider domino tilings
of bounded regions in dimension $n \geq 4$.
We define the {\em twist} of such a tiling,
an elements of $\ZZ/(2)$,
and prove that it is invariant under flips,
a simple local move in the space of tilings.

We investigate which regions $\cD$ are {\em regular}, i.e. whenever two
tilings $\bt_0$ and $\bt_1$ of $\cD \times [0,N]$ have the same twist
then $\bt_0$ and $\bt_1$ can be joined by a sequence of flips provided
some extra vertical space is allowed.  We prove that all boxes
are regular {\em except} $\cD = [0,2]^3$.


Furthermore, given a regular region $\cD$,
we show that there exists a value  $M$
(depending only on $\cD$) such that
if  $\bt_0$ and $\bt_1$ are tilings of equal twist of $\cD \times [0,N]$
then the corresponding tilings can be joined
by a finite sequence of flips in  $\cD \times [0,N+M]$.
As a corollary we deduce that, for regular $\cD$ and large $N$,
the set of tilings of $\cD \times [0,N]$
has two twin giant components under flips,
one for each value of the twist.
\end{abstract}


\section{Introduction}

\footnotetext{2010 {\em Mathematics Subject Classification}.
Primary 05B45; Secondary 52C20, 52C22, 05C70.
{\em Keywords and phrases} Higher dimensional tilings, dominoes, dimers}

A domino in dimension $n$ is a $2\times \overbrace{1\times \cdots
  \times 1}^{n-1}$ rectangular block.  We consider domino tilings of
bounded cubiculated regions in $\mathbb{R}^n$ for $n \geq 4$.  The case $n = 2$ has
been extensively studied, with many remarkable results, see
e.g. \cite{thurston1990}, \cite{Kenyon_lectures}, \cite{survey}.
Almost every question about domino tilings seems to be much harder for $n \ge 3$, see e.g.~\cite{Igor}.

The three dimensional case has distinctive behavior.  
The series of papers \cite{FKMS, primeiroartigo,
  segundoartigo, regulardisk} investigate spaces of three-dimensional tilings, connectivity under local moves, and connections to certain algebraic parameters.  Briefly summarizing:


A \emph{flip} is a local move --  two neighboring parallel dominoes
are removed and placed back in a different position.
Define an equivalence relation on the set $\cT(\cR)$
of domino tilings of a region $\cR$ as
$\bt_0 \approx \bt_1$ if and only if
the tilings $\bt_0$ and $\bt_1$
can be joined by a finite number of flips.

\begin{itemize}
\item If a region $\cR$ of dimension $n=2$
is connected and simply connected
then the equivalence relation is trivial:
for any two tilings $\bt_0, \bt_1$ of the region $\cR$
we have $\bt_0 \approx \bt_1$
(see \cite{thurston1990}).

\item Also for $n = 2$, if a region is planar and connected
but not simply connected
then the {\em flux} is an invariant under flips:
$\bt_0 \approx \bt_1$ if and only if 
$\Flux(\bt_0) = \Flux(\bt_1)$.
More generally, if $\cR$ is a quadriculated surface then
$\Flux(\bt_0) \ne \Flux(\bt_1)$ implies $\bt_0 \not\approx \bt_1$;
if $\Flux(\bt_0) = \Flux(\bt_1)$,
we usually (but not always) have $\bt_0 \approx \bt_1$
(see \cite{saldanhatomei1995}).

\item For $n = 2$, if a region is planar and connected, the number of
  tilings of the region can be enumerated efficiently.  Kasteleyn
  matrices, in particular, provide a linear algebraic approach to the counting problem, (see \cite{Kasteleyn}). 

\item For $n = 3$, if a region is contractible
then the {\em twist} is an integer-valued invariant under flips.
Thus, if $\Tw(\bt_0) \ne \Tw(\bt_1)$ then $\bt_0 \not\approx \bt_1$;
if $\Tw(\bt_0) = \Tw(\bt_1)$,
we usually (but not always) have $\bt_0 \approx \bt_1$ (see
\cite{FKMS, primeiroartigo, segundoartigo}).

\item If $\cR$ is a cubiculated manifold of dimension $3$,
$\Flux(\bt_0)$  is also invariant under flips.
If $\bt_0 \approx \bt_1$ we have
$\Flux(\bt_0) = \Flux(\bt_1)$ and $\Tw(\bt_0) = \Tw(\bt_1)$.
If $\Flux(\bt_0) = \Flux(\bt_1)$ and $\Tw(\bt_0) = \Tw(\bt_1)$
we usually (but not always) have $\bt_0 \approx \bt_1$ (see \cite{FKMS}).

\end{itemize}

In this paper we investigate the above concepts for $n \ge 4$.  There is a fundamental shift in dimensions $4$ and higher.  
In Section~\ref{section:twist} we define the {\em twist} of a tiling which is no longer an element of $\ZZ$ but is naturally
an element of $\ZZ/(2)$. The definition of twist for $n \ge 4$ is in a sense simpler, see
 Lemma \ref{lemma:ZZ2}. 
In Theorem~\ref{lemma:fliptrit} we prove that the twist is invariant under flips.

Sections~\ref{section:plug} and~\ref{section:Delta} are concerned with enumeration and
construct Kasteleyn matrices for the four dimensional case.  As in
\cite{regulardisk}, we focus on {\em cylinders}: regions of the form
$\cR_N = \cD \times [0,N]$ where $\cD \subset \RR^{n-1}$ is a balanced
contractible region.  When $\cD$ is fixed and $N$ goes to infinity, we
prove that the set of tilings $\cT(\cR_N)$ is almost evenly split
between tilings with twists $0$ and $1$ (see Examples
\ref{example:222}, \ref{example:223} and Corollary
\ref{coro:quasibalance}).  This is in contrast to the
three-dimensional case where it is believed that the twist is normally
distributed.  Our result implies, however, that also in dimension $n =
3$, $(\Tw(\bt) \bmod 2)$ is almost evenly split between $0$ and $1$.

If $\bt_0$ and $\bt_1$ are tilings of $\cR_{N_0}$ and $\cR_{N_1}$,
respectively, then $\bt_0$ and $\bt_1$ can be concatenated
to define a tiling $\bt_0 \ast \bt_1$ of $\cR_{N_0 + N_1}$.
If $M$ is even, the region $\cR_{M}$ admits a simple tiling,
the {\em vertical} tiling $\bt_{\nvert,M}$,
formed by dominoes of the form $s \times [k,k+2]$
where $s \subset \cD$ is a unit cube
and $k \in [0,M)$ is an even integer.
Also following \cite{regulardisk},
we define a weaker equivalence relation:
$\bt_0 \sim \bt_1$ if and only if there exists an even integer $M$
such that $\bt_0 \ast \bt_{\nvert,M} \approx \bt_1 \ast \bt_{\nvert,M}$.
Under this equivalence relation,
concatenation defines the {\em domino group} $G_{\cD}$.

Given $\cD$, we  consider the {\em domino complex}, a $2$-complex $\cC_{\cD}$
with a base point $\emptyplug$.
Tilings of $\cR_N$ are interpreted as closed paths of length $N$
in $\cC_{\cD}$, starting and ending at $\emptyplug$.
Two tilings $\bt_0$ and $\bt_1$ satisfy
$\bt_0 \sim \bt_1$ if and only if their paths are homotopic.
Thus, there exists a natural isomorphism
$G_{\cD} \simeq \pi_1(\cC_{\cD};\emptyplug)$
between the domino group and the fundamental group of $\cC_{\cD}$, see Sections~\ref{section:complex} and~\ref{section:hamiltonian}.

For $n \ge 4$,
a region $\cD \subset \RR^{n-1}$ is {\em regular}
if and only if its domino group satisfies
$G_{\cD} \simeq \ZZ/(2) \oplus \ZZ/(2)$.
Equivalently, $\cD$ is regular if and only if
$\Tw(\bt_0) = \Tw(\bt_1)$ implies $\bt_0 \sim \bt_1$
(where $\bt_0$ and $\bt_1$ are tilings of $\cR_N = \cD \times [0,N]$).
For $n = 3$, 
a region $\cD \subset \RR^{n-1}$ is {\em regular}
if and only if $G_{\cD} \simeq \ZZ \oplus \ZZ/(2)$.



In Sections~\ref{section:small}, \ref{section:22L}, \ref{section:box4}
and \ref{section:box5}, we characterize the case for boxes as follows:

\begin{theo}
\label{theo:regularbox}
Let $n \ge 4$ and consider positive integers
$L_1 \ge \cdots \ge L_{n-1} \ge 2$, at least one of them even.
Consider the box $\cD = [0,L_1] \times \cdots \times [0,L_{n-1}]$.
If $n = 4$ and $L_1 = L_2 = L_3 = 2$ then $\cD$ is not regular.
In every other case, $\cD$ is regular.
\end{theo}

\begin{remark}
\label{remark:222-223}
In the irregular case $\cD = [0,2]^3 \subset \RR^3$ 
there exists an isomorphism $G_{\cD} \approx \ZZ \oplus \ZZ/(2)$.
This example is discussed in Example \ref{example:222}
(Section \ref{section:examples})
and the claim is proved in Lemma \ref{lemma:222group}
(Section \ref{section:small}).
The box $\cD = [0,2]^2 \times [0,3] \subset \RR^3$,
on the other hand, is regular:
see Remark \ref{remark:2233}, Example \ref{example:223}
and Lemma \ref{lemma:22L}.
\end{remark}

By definition, if $\cD$ is regular and $\bt_0, \bt_1$ are tilings of $\cR_N$
with $\Tw(\bt_0) = \Tw(\bt_1)$ then
$\bt_0$ and $\bt_1$ can be joined by a finite sequence of flips
provided some extra vertical space $M$ is allowed.
The next result, proved in Section~\ref{section:smallM}, shows that the amount of extra space is bounded
(as a function of $N$).

\begin{theo}
\label{theo:smallM}
Let $n \ge 4$.
If $\cD$ is regular then there exists $M \in 2\NN^\ast = \{2, 4, 6, \ldots\}$
(depending on $\cD$ only) such that:
if $N$ is a positive integer and
$\bt_0, \bt_1$ are tilings of $\cR_N = \cD \times [0,N]$
with $\Tw(\bt_0) = \Tw(\bt_1)$ then
$\bt_0 \ast \bt_{\nvert,M} \approx \bt_1 \ast \bt_{\nvert,M}$.
\end{theo}


\begin{coro}
\label{coro:twin}
Let $\cD \subset \RR^{n-1}$ be a regular region, with $n \ge 4$.
There exist connected components (under $\approx$)
$T_{i} \subset \cT(\cR_N)$, $i \in \ZZ/(2)$, such that
\[
\lim_{N \to \infty} \frac{|T_{0}|}{|\cT(\cR_N)|} =
\lim_{N \to \infty} \frac{|T_{1}|}{|\cT(\cR_N)|} = \frac12,
\qquad
\limsup_{N \to \infty}
\frac{\log|\cT(\cR_N) \smallsetminus (T_{0} \cup T_{1})|}{\log|\cT(\cR_N)|} < 1.
\]
\end{coro}

In other words, the set of tilings of $\cR_N$
has two twin giant components.
There are small components, but their total relative measure
goes to zero exponentially (when $N \to \infty$).

\smallskip

As in the three dimensional case,
it would be interesting to clarify which other contractible regions
(not boxes!) are regular.
It would also be interesting to study the domino group 
in other higher dimensional examples.
We remind the reader that in the three dimensional case
the domino group may have exponential growth.
For instance, if $\cD = [0,2] \times [0,L]$, $L \ge 3$,
we construct in \cite{regulardisk} a surjective homomorphism
from $G_{\cD}^{+}$ (a subgroup of index two of $G_{\cD}$)
to the free group $F_2$.
Does something similar happen in higher dimensions?
Notice that exponential growth of the domino group
implies that all connected components under flips are small
(unlike the situation described in Corollary \ref{coro:twin}).

\smallskip

{\bf Acknowledgments}
The authors thank Juliana Freire, Pedro Milet and Breno Pereira
for helpful conversations, comments and suggestions.
We also thank the referee for several thoughtful and helpful comments.
The second author is also thankful for the generous support of
CNPq, CAPES and FAPERJ (Brazil). 
We also thank the Brown-Brazil Initiative for financial support,
particularly during the visit of the second author to Brown University.


\section{Examples}
\label{section:examples}

In this section we present a few small examples.
In all but the smallest cases, the results were obtained by computer;
only some very small examples can be worked out by hand.
We also show how to draw a tiling $\bt$ of a region $\cR \subset \RR^4$,
particularly if $\cR$ is of the form $\cR = \cR_N = \cD \times [0,N]$,
$\cD \subset \RR^3$.

We first recall how tilings of $\cR_N = \cD \times [0,N] \subset \RR^3$
are drawn in \cite{regulardisk} for $\cD \subset \RR^2$,
$\cD$ a quadriculated disk.
An example is given in Figure \ref{fig:knot} for $\cD = [0,3]^2$,
$\cR = \cD \times [0,2]$.
This region admits $229$ tilings.
The first and last tiling in Figure \ref{fig:knot} admit no flip;
the other $227$ tilings form a single connected component under flips.
A tiling is represented as a sequence of floors;
vertical dominos (i.e., dominoes not contained in a floor)
are represented by two squares, one in each floor.

\begin{figure}[ht]
\begin{center}
\includegraphics[scale=0.275]{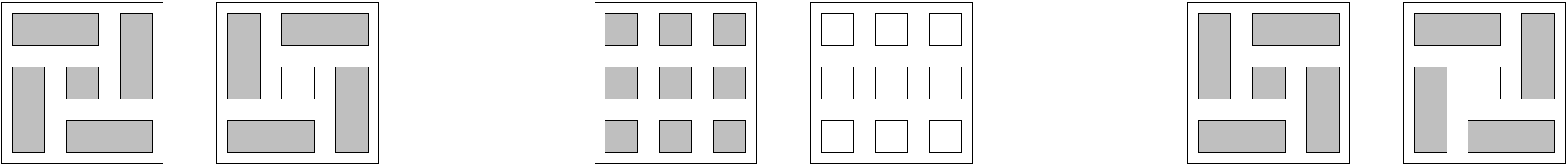}
\end{center}
\caption{Three tilings of the box
$\cR = [0,3]^2 \times [0,2] \subset \RR^3$.}
\label{fig:knot}
\end{figure}

Figures \ref{fig:3322a} and \ref{fig:3322b} show
three and two tilings of the box 
$\cR = [0,3]^2 \times [0,2]^2 \subset \RR^4$, respectively;
Figure \ref{fig:2222} shows two tilings of $\cR = [0,2]^4$.
Let $x_1, \ldots, x_4$ be the coordinates of $\RR^4$.
Each $3\times3$ square in Figure \ref{fig:3322a} represents
a slice of the form $i-1 \le x_3 \le i$, $j-1 \le x_4 \le j$,
$i, j \in \{1,2\}$.
The four squares (slices) are shown in the natural positions:
$i=1$ in the top row, $i=2$ in the bottom row;
$j=1$ in the left column, $j=2$ in the right column.

\begin{figure}[ht]
\begin{center}
\includegraphics[scale=0.275]{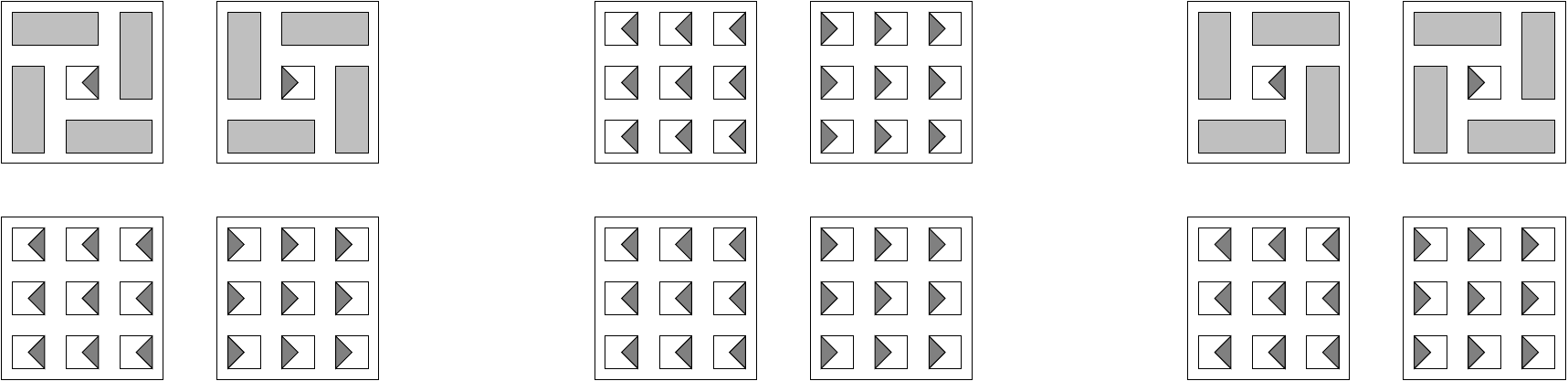}
\end{center}
\caption{Three tilings of the box
$\cR = [0,3]^2 \times [0,2]^2 \subset \RR^4$.}
\label{fig:3322a}
\end{figure}

Dominoes in the directions $x_1, x_2$ are contained in slices
and appear in the figure as dominoes.
Dominoes in the directions $x_3, x_4$ appear as a pair of unit squares,
one in one slice, one in another.
A dark triangle in such unit squares indicates
the position of the partner:
it is as near the partner (in the figure) as possible.
Thus, for instance, the two central unit squares
in each $3\times 3$ square in the top row of the first tiling
in Figure \ref{fig:3322a} form a domino.

\begin{remark}
\label{remark:2233}
The first and third tilings in Figure \ref{fig:3322a}
can be connected by a sequence of $22$ flips.
The reader should contrast this with the fact
that the first and third tilings in Figure \ref{fig:knot}
can not be connected by a sequence of flips,
not even if abundant extra $3$-dimensional space
with vertical dominoes is added around the box.
Indeed, the two tilings $\bt_{\pm 1}$  in Figure \ref{fig:knot}
have twists $\Tw(\bt_1) = +1 \ne -1 = \Tw(\bt_{-1})$
and flips preserve twist \cite{FKMS, regulardisk}.
\end{remark}

\begin{figure}[ht]
\begin{center}
\includegraphics[scale=0.275]{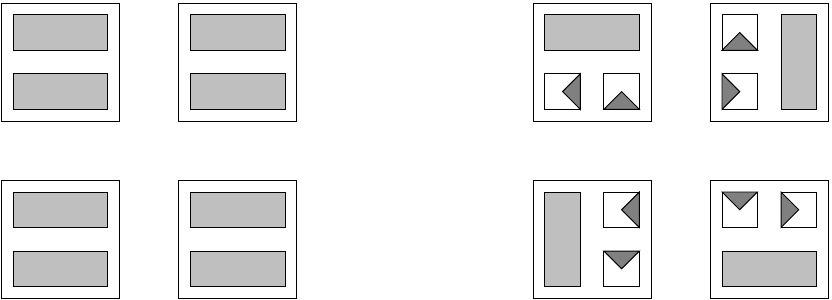}
\end{center}
\caption{Two tilings of the box
$\cR = [0,2]^4 \subset \RR^4$.}
\label{fig:2222}
\end{figure}

\begin{example}
\label{example:222}
The smallest non trivial region is $\cD = [0,2]^3$.
We describe the connected components via flips
of the space of tilings of $\cR = \cD \times [0,N]$.
\begin{itemize}
\item
For $N = 2$ there are $272$ tilings and $9$ components.
The largest component has size $264$ and includes all tilings of twist $0$.
There are $8$ tilings of twist $1$:
each one is isolated, as no flip is possible.
The first tiling in Figure \ref{fig:2222} shows a tiling
in the largest component;
the second tiling is isolated.
\item
For $N = 3$ there are three components:
the largest one has size $5985$ (i.e., includes $5985$ tilings) and twist $0$;
the other two have size $180$ and twist $1$.
\item
For $N = 4$ the components are:
one of size $143065$ and twist $0$;
two of size $6412$ and twist $1$;
$56$ components of sizes $1$ or $2$ and twist $0$.
\item
For $N = 5$ the components are:
one of size $3386376$ and twist $0$;
two of size $202224$ and twist $1$;
two of size $2028$ and twist $0$.
\item
For $N = 6$ the components are:
one of size $80353593$ and twist $0$;
two of size $5987060$ and twist $1$;
two of size $98144$ and twist $0$;
$392$ components of sizes $1$, $2$ or $4$ and twist $1$.
\item
For $N = 30$ the approximate number of tilings of twist $0$ and $1$ are,
respectively, $1.05 \cdot 10^{41}$ and $0.736 \cdot 10^{41}$.
\item
For $N = 50$ the approximate number of tilings of twist $0$ and $1$ are,
respectively, $0.515 \cdot 10^{69}$ and $0.463 \cdot 10^{69}$.
\end{itemize}
As we shall prove in Lemma \ref{lemma:222}, the region $\cD$ is not regular.
This is consistent with the fact that there exist several large components
in the space of tilings of $\cR_N$ for large $N$.
\end{example}

\begin{example}
\label{example:223}
We now consider $\cD = [0,2]^2 \times [0,3]$
and tilings of $\cR_N = \cD \times [0,N]$.
\begin{itemize}
\item
For $N = 3$ the components are:
one of size $762572$ and twist $0$
($T_0$ in the notation of Corollary \ref{coro:twin});
one of size $99280$ and twist $1$ ($T_1$);
$16$ of size $16$ and twist $0$;
$2$ of size $2$ and twist $0$.
Up to the obvious identification between
$\cR_3 = [0,2]^2 \times [0,3]^2$ and $[0,3]^2 \times [0,2]^2$,
the tilings in Figures \ref{fig:3322a} and \ref{fig:3322b}
are tilings of $\cR_3$.
The second tiling in Figure \ref{fig:3322a} belongs to $T_0$.
The first and third tilings in Figure \ref{fig:3322a} both belong to $T_1$
(see Remark \ref{remark:2233}).
Figure \ref{fig:3322b} shows two tilings in components of sizes $16$ and $2$.
\item
For $N = 4$ the components are:
one of size $106303993$ and twist $0$ ($T_0$);
one of size $20723112$ and twist $1$ ($T_1$);
$8$ of size $49$ and twist $0$;
$16$ of size $16$ and twist $1$;
$16$ of size $1$ and twist $1$.
\item
For $N = 30$ the approximate number of tilings of twist $0$ and $1$ are,
respectively, $0.117 \cdot 10^{65}$ and $0.108 \cdot 10^{65}$.
\end{itemize}

\begin{figure}[ht]
\begin{center}
\includegraphics[scale=0.275]{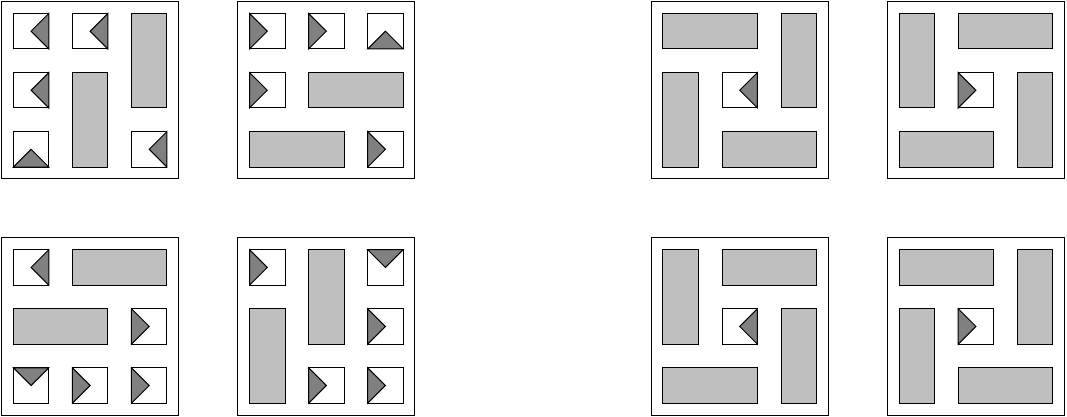}
\end{center}
\caption{Two tilings of the box
$\cR = [0,3]^2 \times [0,2]^2 \subset \RR^4$.}
\label{fig:3322b}
\end{figure}

As we shall prove in Lemma \ref{lemma:22L},
the box $\cD$ is regular.
This is consistent with the fact that there exist exactly two large components.
For large $N$, the two large components have approximately the same size.
\end{example}

\begin{remark}
\label{remark:source}
The programs used to verify these examples were written in C or
{{C\nolinebreak[4]\hspace{-.05em}\raisebox{.4ex}{\tiny\bf ++}}};
sources available on the home page of the authors \cite{4dominosource}.
Most of them were written by the authors;
some older programs were written by our collaborators
from previous publications, including J. Freire and P. Milet,
and by students, including B.~Pereira.
The cases $N \le 6$ in Example \ref{example:222}
and the cases $N \le 4$ in Example \ref{example:223}
were performed by brute force.
Tilings are encoded by strings of characters:
each character corresponds to a unit cube and indicates
the direction of the corresponding domino.
We first produced a list of all tilings in alphabetical order
and then computed the connected components.
The cases of larger $N$ are too large for a brute force approach.
We then use the theory
explained in \cite{probabledomino}
and in Sections \ref{section:twist} and \ref{section:plug} below.
In particular, in both examples we explicitly compute both
the adjacency matrix $A$ (as in Definition \ref{defin:adjacency})
and the matrix $\tilde A$ (as in Equation \eqref{equation:Delta}).
With the use of the arbitrary precision library gmp,
this allows us to obtain the exact number of tilings
with each value of the twist.
Computing the sizes of the connected components
appears to be significantly harder.
\end{remark}


\section{Twist}
\label{section:twist}

For the remainder of the paper, unless otherwise stated,
by a region $\cR$, we mean
a balanced cubiculated subset of $\RR^n$.

Given a region $\cR$ let $\cT(\cR)$ be the set of domino tilings of $\cR$.
Construct a simple bipartite graph $\cG_{\cR}$ as follows.
Vertices of $\cG_{\cR}$ are unit cubes in $\cR$
and two vertices of $\cG_{\cR}$ are joined by an edge if and only if
the two corresponding unit cubes share a face of codimension one.
We assume the vertices of $\cG_{\cR}$  belong to $\ZZ^n \subset \RR^n$
and the edges of $\cG_{\cR}$ are unit segments.
The color of a vertex $v = (x_1,\ldots,x_n) \in \ZZ^n$ is given by
$(-1)^{x_1+\cdots +x_n}$
(black is $+1$, white is $-1$).
Let $b \in \NN^\ast = \{1, 2, 3, \ldots \}$
be the number of black vertices of $\cG_{\cR}$;
recall that we assume that $\cR$ is balanced
so that $\cG_{\cR}$ also has $b$ white vertices.
Label the black vertices as $v_1, \ldots, v_b$ and
the white vertices as $w_1, \ldots, w_b$.
Tilings $\bt \in \cT(\cR)$ correspond to perfect matchings of $\cR$,
or, equivalently,
to bijections $\sigma_\bt: \{1,2,\ldots,b\} \to \{1,2,\ldots,b\}$
such that $v_i$ is adjacent to $w_{\sigma(i)}$ (for all $i$).

The adjacency matrix of $\cR$ is an indicator matrix recording if $v_i$ and $w_j$ are adjacent.
Tilings of $\cR$ naturally correspond to
nonzero terms in the expansion of the determinant of the adjacency matrix.
We imitate Kasteleyn's construction
(originally for dimension $2$ \cite{Kasteleyn})
to define a matrix $K \in \ZZ^{b \times b}$ with entries
$K_{ij} \in \{+1, -1\}$ if $v_i$ and $w_j$ are adjacent and 
$K_{ij} = 0$ otherwise.
As above, consider $v_i, w_j \in \ZZ^n$ so that
$v_i - w_j = \pm e_k$ for some $k \in \{1,2,\ldots,n\}$.
Write $v_i = (x_1,\ldots,x_n) \in \ZZ^n$ so that
$w_j = (x_1,\ldots,x_k\pm 1,\ldots,x_n)$ and set
\begin{equation}
\label{equation:K}
K_{ij} = (-1)^{x_1+\cdots+x_{k-1}}.
\end{equation}
As in the case of  the adjacency matrix,
tilings of $\cR$ naturally correspond to
nonzero terms in the expansion of $\det(K)$.
Given a tiling $\bt \in \cT(\cR)$,
define the signed permutation matrix $T_{\bt}$
by $(T_{\bt})_{i,\sigma_{\bt}(i)} = K_{i,\sigma_{\bt}(i)} \in \{+1,-1\}$
and $(T_{\bt})_{ij} = 0$ otherwise.

\begin{defin} The {\em twist} $\Tw(\bt) \in \ZZ/(2)$ is defined  by
\label{defin:twist}
\begin{equation}
\label{equation:twist}
\det(T_{\bt}) = (-1)^{\Tw(\bt)} =
\sign(\sigma_{\bt}) \prod_i K_{i,\sigma_{\bt}(i)}. 
\end{equation}
\end{defin}

\begin{defin}

The {\em defect} $\Delta(\cR)$ of a region $\cR$:
\begin{align}
\label{equation:tw}
\Delta(\cR) &= \det(K) = \sum_{\bt \in \cT(\cR)} (-1)^{\Tw(\bt)} \\
\notag
&=
|\{ \bt \in \cT(\cR) \;|\; \Tw(\bt) = 0 \}| -
|\{ \bt \in \cT(\cR) \;|\; \Tw(\bt) = 1 \}|.
\end{align}
Here $\sign(\sigma) = (-1)^{\inv(\sigma)}$ is the sign
of the permutation $\sigma$,
$\inv(\sigma) = |\Inv(\sigma)|$ is the number of inversions of $\sigma$
and $\Inv(\sigma)$ is the set of inversions of $\sigma$.

\end{defin}

Therefore the determinant of our Kasteleyn matrix does not enumerate the total number of tilings.  Instead it detects the difference between the number of tilings with twist $0$ and twist $1$.

\begin{remark}
\label{remark:labelsign}
The definition of $\Tw$ depends on the labeling
of the vertices of $\cG_{\cR}$.
Changing labelings corresponds to permuting rows and columns of $K$
and therefore possibly changing the value of $\Tw(\bt) \in \ZZ/(2)$
for all tilings $\bt$.
\end{remark}

A {\em flip} is a local move:
remove two adjacent parallel dominoes and
place them back in the only other possible way.
A {\em trit} is another local move.
Consider a block formed by $8$ unit cubes,
of dimensions $2 \times 2 \times 2 \times 1 \times \cdots \times 1$:
if we remove from the block two opposite unit cubes,
we are left with the union of six unit cubes,
which can be tiled by dominoes in precisely two ways.
A trit consists in finding three dominoes in the configuration above,
removing them and placing them back in the only other possible way.
The trit is therefore a local move involving three dominoes.
All other local moves involving three (or fewer) dominoes
reduce to a (very short) sequence of flips,
as can be verified case by case.
The trit does not reduce to a sequence of flips,
as shown in the next Theorem.

\begin{theo}
\label{lemma:fliptrit}
Let $\bt_0$ and $\bt_1$ be tilings of a region $\cR$.
If $\bt_0$ and $\bt_1$ differ by a flip then $\Tw(\bt_1) = \Tw(\bt_0)$.
If $\bt_0$ and $\bt_1$ differ by a trit then $\Tw(\bt_1) = 1-\Tw(\bt_0)$.
\end{theo}

\begin{proof}
A flip is always contained in a plane, an affine subspace of dimension $2$.
Assume the two relevant dimensions are  $k_0$ and $k_1$
with $1 \le k_0  < k_1 \le n$.
Among the four vertices involved, let $v = (x_1, \ldots, x_n)$
be the one with smallest coordinates
so that the other three vertices are
$v + e_{k_0}$, $v + e_{k_1}$ and $v + e_{k_0} + e_{k_1}$.
Without loss of generality, suppose
$\bt_0$ contains the dominoes
$(v,v+e_{k_0})$ and $(v+e_{k_1},v+e_{k_0}+e_{k_1})$ and 
$\bt_1$ contains the dominoes
$(v,v+e_{k_1})$ and $(v+e_{k_0},v+e_{k_0}+e_{k_1})$.

Each tiling defines a bijection from the set of black vertices
to the set of white vertices.
The two bijections corresponding to $\bt_0$ and $\bt_1$
differ by a transposition (on either side)
and therefore have opposite parities.
Also, the matrix $K$ assigns signs to edges.
The signs assigned to 
$(v,v+e_{k_0})$ and $(v+e_{k_1},v+e_{k_0}+e_{k_1})$ both equal
$(-1)^{x_1+\cdots+x_{k_0-1}}$ and are therefore equal.
The signs assigned to 
$(v,v+e_{k_1})$ and $(v+e_{k_0},v+e_{k_0}+e_{k_1})$ are
$(-1)^{x_1+\cdots+x_{k_0}+ \cdots +x_{k_1-1}}$
and
$(-1)^{x_1+\cdots+(x_{k_0}+1)+ \cdots +x_{k_1-1}}$
and are therefore different.
We thus have 
$\det(T_{\bt_1}) = \det(T_{\bt_0})$
and therefore
$\Tw(\bt_1) = \Tw(\bt_0)$, proving the first claim.

A trit is always contained in an affine subspace of dimension $3$.
Assume the relevant dimensions to be $k_0, k_1, k_2$
with $1 \le k_0  < k_1 < k_2 \le n$.
Assume that 
the three dominoes
$(v+e_{k_0},v+e_{k_0}+e_{k_1})$,
$(v+e_{k_1},v+e_{k_1}+e_{k_2})$ and
$(v+e_{k_2},v+e_{k_2}+e_{k_0})$
are contained in $\bt_0$ and that
the three dominoes
$(v+e_{k_0},v+e_{k_0}+e_{k_2})$,
$(v+e_{k_1},v+e_{k_1}+e_{k_0})$ and
$(v+e_{k_2},v+e_{k_2}+e_{k_1})$
are contained in $\bt_1$.
As above, write $v = (x_1,\ldots,x_n)$.

The bijections corresponding to $\bt_0$ and $\bt_1$ now differ
by a $3$-cycle and therefore have the same parity.
The signs assigned by $K$ to
$(v+e_{k_2},v+e_{k_2}+e_{k_0})$ and
$(v+e_{k_1},v+e_{k_1}+e_{k_0})$ are both
$(-1)^{x_1+\cdots+x_{k_0-1}}$ and therefore equal.
The signs assigned to
$(v+e_{k_2},v+e_{k_2}+e_{k_1})$ and
$(v+e_{k_0},v+e_{k_0}+e_{k_1})$
are $(-1)^{x_1+\cdots+x_{k_0}+\cdots+x_{k_1-1}}$
and $(-1)^{x_1+\cdots+(x_{k_0}+1)+\cdots+x_{k_1-1}}$,
respectively, and therefore different.
Finally, the signs assigned to
$(v+e_{k_1},v+e_{k_1}+e_{k_2})$ and
$(v+e_{k_0},v+e_{k_0}+e_{k_2})$ are
\[ (-1)^{x_1+\cdots+x_{k_0}+\cdots+(x_{k_1}+1)+\cdots+x_{k_2-1}}, \quad
(-1)^{x_1+\cdots+(x_{k_0}+1)+\cdots+x_{k_1}+\cdots+x_{k_2-1}}, \]
respectively, and therefore equal.
We thus have 
$\det(T_{\bt_1}) = -\det(T_{\bt_0})$
and therefore
$\Tw(\bt_1) = 1-\Tw(\bt_0)$, proving the second claim.
\end{proof}


Before moving on,  we show the naturality of the matrix $K$.

Consider a region $\cR \subset \RR^n$ and
its graph $\cG_{\cR}$.
A {\em Kasteleyn system} for $\cR$ assigns to each edge of $\cG_{\cR}$
a coefficient $+1$ or $-1$ satisfying the following condition:
if four edges form a square then the product of their coefficients is $-1$.
Given a Kasteleyn system we also have a {\em Kasteleyn matrix} for $\cR$,
a matrix $\tilde K \in \ZZ^{b \times b}$:
if $v_i$ and $w_j$ are adjacent then $\tilde K_{ij}$ 
is the coefficient of the edge $v_iw_j$
(and the coefficients form a Kasteleyn system).
Thus, for all $i,j$, $\tilde K_{ij} \in \{+1,-1\}$ if and only if
$v_i$ and $w_j$ are adjacent.
Also, for all $i_0, i_1, j_0, j_1$, we have
$\tilde K_{i_0j_0} \tilde K_{i_0j_1} \tilde K_{i_1j_0} \tilde K_{i_1j_1}
\in \{0,-1\}$.

The matrix $K$ is an example of a Kasteleyn matrix.
The following lemma shows that if $\cR$ is connected and simply connected
then any other Kasteleyn matrices are only minor variations.

\begin{lemma}
\label{lemma:kasteleyn}
Consider a region $\cR \subset \RR^n$.
Assume furthermore that $\cR$ is connected and simply connected.
Then $\tilde K$ is a Kasteleyn matrix if and only if there
exist diagonal matrices $D_{\black}, D_{\white}$
with diagonal entries equal to $\pm 1$
and $\tilde K = D_{\black} K D_{\white}$.
\end{lemma}

\begin{proof}
It is straightforward to verify that if $D_{\black}, D_{\white}$
are diagonal matrices
with diagonal entries equal to $\pm 1$ then
$D_{\black} K D_{\white}$ is indeed a Kasteleyn matrix.
In order to prove the converse we use the language of homology.
Consider a Kasteleyn matrix $\tilde K$ and its corresponding Kasteleyn system.
A Kasteleyn system defines an element of $\alpha \in C^1(\cR;\ZZ/(2))$:
if $e$ is an edge then the coefficient of $e$ 
in the Kasteleyn system is $(-1)^{\alpha(e)}$,
$\alpha(e) \in \ZZ/(2)$.
(Here $C^1(\cR;\ZZ/(2))$ is the first cochain group
of the cell complex $\cR$ with coefficients in $\ZZ/(2)$.)
Let $\alpha, \tilde\alpha \in C^1(\cR;\ZZ/(2))$ correspond to 
the original $K$ and to $\tilde K$, respectively.
By definition, if $s$ is an oriented square then
$\alpha(\partial s) = \tilde\alpha(\partial s) = 1 \in \ZZ/(2)$.
(Here $\partial: C_2 \to C_1$ is the boundary map.)
Thus $(\alpha - \tilde\alpha)(\partial s) = 0$ for all $s$
and $\alpha - \tilde\alpha \in Z^1$ (i.e., it is closed).
Since $\cR$ is simply connected we have
from the universal coefficient theorem that $H^1(\cR;\ZZ/(2)) = 0$:
it follows that $\alpha - \tilde\alpha \in B^1$ (i.e., it is exact).
(Here $B^1 \subseteq C^1$ is the image of
the coboundary map $\partial^\ast: C^0 \to C^1$.)
In other words, there exists $\delta \in C^0(\cR;\ZZ/(2))$ with
$\alpha - \tilde\alpha = \partial^\ast\delta$.
For any edge $e = vw$ we have
$\alpha(e) - \tilde\alpha(e) = \delta(w) - \delta(v)$.
Thus, $\delta$ gives us the desired
diagonal matrices $D_{\black}$ and $D_{\white}$.
\end{proof}

\begin{coro}
Consider a connected and simply connected region
$\cR \subset \RR^n$, $n \ge 3$.
Consider a fixed Kasteleyn matrix $\tilde K$ for $\cR$.
For a tiling $\bt$ of $\cR$, construct a signed permutation matrix
$\tilde T = T_{\bt,\tilde K}$ with nonzero entries 
$\tilde T_{ij} = \tilde K_{ij}$ when $v_i$ and $w_j$ form a domino of $\bt$.
Then there exists $\varepsilon \in \{+1,-1\}$ such that
for all $\bt$ we have $\det(T_{\bt,\tilde K}) = \varepsilon\Tw(\bt)$.
\end{coro}

\begin{proof}
By construction, $\Tw(\bt) = \det(T_{\bt,K})$ 
for the original Kasteleyn matrix $K$.
From Lemma \ref{lemma:kasteleyn}, there exist diagonal matrices
$D_{\black}$ and $D_{\white}$ with $\tilde K = D_{\black}KD_{\white}$.
By construction we also have
$T_{\bt,\tilde K}  = D_{\black}T_{\bt,K} D_{\white}$.
Take  $\varepsilon = \det(D_{\black}D_{\white})$:
we have $\det(T_{\bt,\tilde K}) = \varepsilon \det(T_{\bt,K})$
for all $\bt$, as desired.
\end{proof}


In dimension $n = 3$, the twist $\Tw(\bt)$ is defined
to be an {\em integer}
(see \cite{segundoartigo}, \cite{FKMS}).
In order to avoid confusion, we temporarily write, for $n = 3$,
$\Tw_{\ZZ}$ for the twist as defined in the other references
and $\Tw_{\ZZ/(2)}$ for the twist as defined here.
The following lemma clarifies the relationship
between the two concepts.

\begin{lemma}
\label{lemma:ZZ2}
Let $\cD \subset \RR^2$ be a balanced quadriculated disk.
Let $\cR_N = \cD \times [0,N] \subset \RR^3$.
For any tiling $\bt$ of $\cR_N$ we have
$\Tw_{\ZZ/(2)}(\bt) = (\Tw_{\ZZ}(\bt) \bmod 2)$.
\end{lemma}

The proof below relies heavily on notation,
definitions and results from \cite{saldanha2002} and \cite{FKMS}.
We feel that providing a more self-contained exposition
would imply too much repetition.

\begin{proof}
Here, $\Tw_{\ZZ}$ is given by Definition 7.7 from \cite{FKMS}.
Since $\cR_N$ is contractible the flux is $0$ and $m = 0$.
Given two tilings $\bt_0$ and $\bt_1$ which differ by a cycle,
Definition 7.2 gives us
$\Tw_{\ZZ}(\bt_1) - \Tw_{\ZZ}(\bt_0) = \phi(t_1;t_1-t_0)$.
It thus suffices to check that
$\Tw_{\ZZ/(2)}(\bt_1) - \Tw_{\ZZ/(2)}(\bt_0) = (\phi(t_1;t_1-t_0) \bmod 2)$.
If there exists a Seifert surface for the cycle $t_1-t_0$
then this follows from Kasteleyn systems,
as discussed in \cite{saldanha2002}.
More generally, we may take refinements, as in \cite{FKMS}.
\end{proof}


\section{Plugs and floors}
\label{section:plug}

For the remainder of the paper, all regions $\cD \subset \RR^{n-1}$
are assumed to be balanced, cubiculated and contractible.  

Consider a region $\cD \subset \RR^{n-1}$:
we are interested in tilings of $\cR_N = \cD \times [0,N]$.
We imitate some of the constructions 
from \cite{regulardisk}, where the case $n = 3$ is discussed.

A domino $d$ (of dimension $n$) contained in $\cR_N$ is {\em horizontal} 
if it is of the form $\tilde d \times [k-1,k]$
where $\tilde d \subset \cD$ is a domino (of dimension $n-1$).
A domino $d \subset \cR_N$ is {\em vertical} otherwise, i.e.,
if it is of the from $s \times [k-1,k+1]$
where $s \subset \cD$ is a unit cube.
A {\em plug} is a balanced set of unit cubes contained in $\cD$
(or balanced set of vertices in $\cG_{\cD}$).
This includes the empty plug $\emptyplug = \emptyset$ and
its complement $\fullplug = \cD$.
Let $\cP = \cP_{\cD}$ be the set of all plugs.
Two plugs $p_0, p_1 \in \cP$ are {\em disjoint}
if and only if $p_0 \cap p_1 = \emptyplug$.
A {\em floor} is a triple $(p_0,f,p_1)$ where
$p_0, p_1 \in \cP$ are disjoint plugs and $f$ is a domino tiling
of $\cD_{p_0,p_1} = \cD \smallsetminus (p_0 \cup p_1)$.
A tiling of $\cR_N$ can be identified with a alternating sequence
of plugs and floors:
\begin{equation}
\label{equation:plugsandfloors}
\bt = (p_0 = \emptyplug, \bff_1, p_1, \ldots,
p_{N-1}, \bff_N, p_N = \emptyplug).
\end{equation}
Here $p_k$ is the set of unit cubes $s \subset \cD$
such that the vertical domino $s \times [k-1,k+1]$
is contained in $\bt$.
Also, $\bff_k = (p_{k-1}, f_k, p_k)$ where $f_k$ consists
of dominoes $\tilde d \subset \cD$ such that
the horizontal domino $d \times [k-1,k]$ is contained in $\bt$.

The \emph{domino complex} $\cC_{\cD}$ is
a $2$-complex associated to $\cD$.
We first construct a graph $\cC_{1,\cD}$
which is essentially the $1$-skeleton of $\cC_{\cD}$;
the complex itself will be constructed in Section \ref{section:complex}.
The set of vertices of $\cC_{1,\cD}$ is the set of plugs $\cP$.
If $p_0, p_1 \in \cP$ are not disjoint there is no edge joining them. 
If $p_0$ and $p_1$ are disjoint there is one edge for every tiling
of $\cD_{p_0,p_1}$.
Thus, a floor $\bff_1 = (p_0,f_1,p_1)$ is identified with
an edge joining $p_0$ and $p_1$.  
Each tiling of $\cD = \cD_{\emptyplug,\emptyplug}$
yields a loop based on the vertex $\emptyplug$;
these are the only loops in $\cC_{1,\cD}$.

Each tiling $\bt$ of $\cR_N$ is identified with a closed walk
of length $N$ in $\cC_{1,\cD}$ from $\emptyplug$ to itself.
More generally, consider the \emph{cork}
\[ \cR_{0,N;p_0,p_N} = \cR_N \smallsetminus \left(
(p_0 \times [0,1]) \cup (p_N \times [N-1,N]) \right). \]
Each tiling of $\cR_{0,N;p_0,p_N}$ is identified
with a walk in $\cC_{\cD}$, of length $N$,
starting at $p_0$ and ending at $p_N$.
In order to count tilings, we   construct 
the adjacency matrix $A$ of $\cC_{1,\cD}$.

\begin{defin}
\label{defin:adjacency}
The \emph{adjacency matrix} $A
\in \ZZ^{\cP\times\cP}$ of $\cC_{1,\cD}$ is the matrix given by:
\[
A_{p_0,p_1} =
\begin{cases}
|\cT(\cD_{p_0,p_1})|, & p_0 \cap p_1 = \emptyplug, \\
0, & p_0 \cap p_1 \ne \emptyplug.
\end{cases}
\]
\end{defin}
%

Thus, $(A^N)_{p_0,p_N}$ is the number of tilings of $\cR_{0,N;p_0,p_N}$.
In particular, $$|\cT(\cR_N)| = (A^N)_{\emptyplug,\emptyplug}.$$

In order to compute the defect $\Delta(\cR_N)$
we first define $\tw_{p_0,p_1}(\bt) \in \ZZ/(2)$
for a tiling $\bt$ of $\cD_{p_0,p_1}$.



Label the unit cubes of $\cD$:
the black cubes are $v_1, \ldots, v_b$;
the white cubes are $w_1, \ldots, w_b$.
Construct a Kasteleyn matrix $K$ for $\cD$ as above.
The plug $p_i \in \cP$ consists of
$b_i$ black unit cubes and $b_i$ white unit cubes,
thus defining two subsets
$P_{i,\black}, P_{i,\white} \subseteq \{1,\ldots,b\}$
with $|P_{i,\black}| = |P_{i,\white}| = b_i$:
$j \in P_{i,\black}$ if and only if $v_j$ is contained in $p_i$
(and similarly for white).
If $p_0$ and $p_1$ are disjoint
then $P_{0,\black}$ and $P_{1,\black}$ are disjoint
and so are $P_{0,\white}$ and $P_{1,\white}$.
Define subsets
$D_{p_0,p_1,\black}, D_{p_0,p_1,\white} \subseteq \{1,\ldots,b\}$
and functions $h_{\black}, h_{\white}: \{1,\ldots,b\} \to \{0,\pm 1\}$ by
\begin{gather*}
D_{p_0,p_1,\black} = 
\{1,\ldots,b\} \smallsetminus (P_{0,\black} \cup P_{1,\black}), \quad
D_{p_0,p_1,\white} = 
\{1,\ldots,b\} \smallsetminus (P_{0,\white} \cup P_{1,\white}), \\
h_{\black}(i) = [i \in P_{1,\black}] - [i \in P_{0,\black}], \quad
h_{\white}(i) = [i \in P_{1,\white}] - [i \in P_{0,\white}],
\end{gather*}
so that, for instance,
$i \in D_{p_0,p_1,\black}$ if and only if $h_{\black}(i) = 0$;
we use here Iverson's notation.
Define the subsets $\Inv_{\black}, \Inv_{\white} \subseteq  \{1,\ldots,b\}^2$
and non negative integers
$\inv_{\black,p_0,p_1}, \inv_{\white,p_0,p_1} \in \NN$ by
\[ \Inv_{\ast} = \{(i_0,i_1) \in \{1,\ldots,b\}^2 \;|\;
i_0 < i_1, h_{\ast}(i_0) > h_{\ast}(i_1)\},
\quad
\inv_{\ast,p_0,p_1} = |\Inv_{\ast}|. \]
A tiling $\bt \in \cT(\cD_{p_0,p_1})$ is defined by a bijection
$\sigma_{\bt}: D_{p_0,p_1,\black} \to D_{p_0,p_1,\white}$
such that $K_{i,\sigma_{\bt}(i)} \in \{+1,-1\}$
for all $i \in D_{p_0,p_1,\black}$.
Define the subset $\Inv(\sigma_{\bt}) \subseteq D_{p_0,p_1,\black}^2$
and the non negative integer $\inv(\sigma_{\bt}) \in \NN$ by
\[ \Inv(\sigma_{\bt}) = \{(i_0,i_1) \in D_{p_0,p_1,\black}^2 \;|\;
i_0 < i_1, \sigma_{\bt}(i_0) > \sigma_{\bt}(i_1)\}, 
\quad
\inv(\sigma_{\bt}) = |\Inv(\sigma_{\bt})|. \]
Finally, for $\bt \in \cT(\cD_{p_0,p_1})$, define
$\tw_{p_0,p_1}(\bt), \tk(\bt) \in \ZZ/(2)$ by
\begin{align}
\label{equation:xtw}
\tw_{p_0,p_1}(\bt) &=
\left( 
\tk(\bt) + \inv(\sigma_{\bt}) + \inv_{\black,p_0,p_1} + \inv_{\white,p_0,p_1}
\right) \bmod 2, \\ 
\notag
(-1)^{\tk(\bt)} &=
\prod_{i \in D_{p_0,p_1,\black} } K_{i,\sigma_{\bt}(i)}.
\end{align}

Let $\cD \subset \RR^{n-1}$.  
Let $\tilde A \in \ZZ^{\cP \times \cP}$ be defined by
\begin{equation}
\label{equation:Delta}
\tilde A_{p_0,p_1} =
\sum_{\bt \in \cT(\cD_{p_0,p_1})} (-1)^{\tw_{p_0,p_1}(\bt)},
\end{equation}
if $p_0$ and $p_1$ are not disjoint then $\tilde A_{p_0,p_1} = 0$.

\begin{lemma}
\label{lemma:Delta}
For $\tilde A$ as defined in Equation \ref{equation:Delta} and $N \in \NN$,
\[ \Delta(\cR_N) = (\tilde A^N)_{\emptyplug,\emptyplug}. \]
\end{lemma}

\begin{proof}
Write $\bt \in \cT(\cR_N)$ as a sequence of plugs and floors,
as in Equation \ref{equation:plugsandfloors}.
We claim that
\[ \Tw(\bt) = \sum_{1 \le k \le N} \tw_{p_{k-1},p_k}(f_k), \]
which completes the proof.
For this, we go back to the definition of $\Tw(\bt)$
in Equation \ref{equation:twist} and compute
$\inv(\sigma_{\bt})$ and $\kappa = \prod_i K_{i,\sigma_{\bt}(i)}$.
Let $b_k$ be the number of black unit cubes in $p_k$:
we claim that
\begin{equation}
\label{equation:inv}
\inv(\sigma_{\bt}) - \sum_{1 \le k \le N} 
\left( 
\inv(\sigma_{f_k}) + \inv_{\black,p_{k-1},p_k} + \inv_{\white,p_{k-1},p_k}
\right) =
\sum_{1 \le k < N} b_k^2.
\end{equation}

Let $b$ be the number of black unit cubes in $\cD$.
First label black and white unit cubes in $\cD$.
Next label unit cubes in $\cR_N$,
using the previous labels in each floor
and proceeding by increasing floor.
In particular, for $i \in \{1,\ldots,Nb\}$,
both the $i$-th black and white unit cubes 
are contained in floor $\lceil i/b \rceil \in \{1,\ldots,N\}$.
We therefore have
$| \lceil i/b \rceil - \lceil \sigma_{\bt}(i)/b \rceil | \le 1$
for all $i$.

Recall that an inversion for $\sigma_{\bt}$ is a pair $(i_0,i_1)$
of indices for black unit cubes such that $i_0 < i_1$
and $j_0 = \sigma_{\bt}(i_0) > \sigma_{\bt}(i_1) = j_1$.
We thus have $\lceil i_0/b \rceil \le \lceil i_1/b \rceil$
and $\lceil j_0/b \rceil \ge \lceil j_1/b \rceil$.
We consider the possible cases.
\begin{itemize}
\item{If 
$\lceil i_0/b \rceil = \lceil i_1/b \rceil$
and $\lceil j_0/b \rceil = \lceil j_1/b \rceil$
we may write $k = \lceil i_0/b \rceil = \lceil j_0/b \rceil$.
Both dominoes are then contained in floor $k$,
and the inversion $(i_0,i_1)$ is counted once by 
$\inv(\sigma_{\bt})$ and once by $\inv(\sigma_{f_k})$.
Strictly speaking, in the second case the inversion
is now called $(i_0 - (k-1)b, i_1 - (k-1)b)$,
but we shall not follow such relabelings from now on.}
\item{If $k = \lceil i_0/b \rceil = \lceil i_1/b \rceil$
and $\lceil j_0/b \rceil > \lceil j_1/b \rceil$
then the inversion $(i_0,i_1)$ is counted once by $\inv(\sigma_{\bt})$
and once by $\inv_{\black,p_{k-1},p_k}$.}
\item{If $k = \lceil j_0/b \rceil = \lceil j_1/b \rceil$
and $\lceil i_0/b \rceil < \lceil i_1/b \rceil$
then the inversion $(i_0,i_1)$ is counted once by $\inv(\sigma_{\bt})$
and once by $\inv_{\white,p_{k-1},p_k}$
(in the second case it is called
$(j_1 - (k-1)b,j_0 - (k-1)b)$).}
\item{Finally, if
$k = \lceil i_0/b \rceil < \lceil i_1/b \rceil = k+1$
and 
$k = \lceil j_1/b \rceil < \lceil j_0/b \rceil = k+1$
then the inversion $(i_0,i_1)$ is counted once by $\inv(\sigma_{\bt})$
and not counted by the summation on the left hand side.
For each $k$, there exist $b_k^2$ such inversions,
completing the proof of Equation \ref{equation:inv}.}
\end{itemize}

Write $\kappa = \kappa_{\hz} \kappa_{\vt}$,
where
\[ \kappa_{\hz} =
\prod_{\lceil i/b \rceil = \lceil \sigma_{bt}(i)/b \rceil}
K_{i,\sigma_{\bt}(i)}, \qquad
\kappa_{\vt} =
\prod_{\lceil i/b \rceil \ne \lceil \sigma_{bt}(i)/b \rceil}
K_{i,\sigma_{\bt}(i)}. \]
For each $k$, $1 \le k < N$, we have
\[ 
\prod_{\{ \lceil i/b \rceil, \lceil \sigma_{bt}(i)/b \rceil \} = \{ k,k+1 \} }
K_{i,\sigma_{\bt}(i)} = (-1)^{b_k} \]
and therefore $\kappa_{\vt} = (-1)^{(b_1 + \cdots + b_{N-1})}$.
For $1 \le k \le N$, let
\[ \kappa_k = (-1)^{\tk(f_k)} =
\prod_{i \in D_{p_{k-1},p_k,\black}} K_{i,\sigma_{f_k}(i)} \]
so that $\kappa_{\hz} = \prod_k \kappa_k$ and therefore
$\kappa \cdot (\prod_k \kappa_k) = (-1)^{(b_1 + \cdots + b_{N-1})}$.
The desired result now follows from Equation \ref{equation:inv}
and the fact that $b_k \equiv b_k^2 \pmod 2$.
\end{proof}


\begin{lemma}
\label{lemma:tildeA}
Let $\cD \subset \RR^{n-1}$.  
Let $\cP$ be the set of plugs for $\cD$.
Let $\tilde A \in \ZZ^{\cP \times \cP}$ be the matrix defined
in Equations \ref{equation:Delta} and \ref{equation:xtw}.
Then $\tilde A$ is real symmetric.
\end{lemma}

\begin{proof}
Let $p_0, p_1 \in \cP$ be disjoint plugs;
let $\bt \in \cT(\cD_{p_0,p_1})$ be a tiling.
We prove that $\tw_{p_0,p_1}(\bt) = \tw_{p_1,p_0}(\bt)$.
Indeed, the definitions of $\tk(\bt)$ and of $\inv(\sigma_{\bt})$
are unchanged, so it suffices to prove that
$\inv_{\black,p_0,p_1} + \inv_{\black,p_1,p_0} =
\inv_{\white,p_0,p_1} + \inv_{\white,p_1,p_0}$.
Indeed, $\inv_{\black,p_0,p_1} + \inv_{\black,p_1,p_0} =
b_0 b_1 + (b_0 + b_1) (b - b_0 - b_1)$ since it counts pairs  
$\{i_0,i_1\} \subseteq \{1,\ldots,b\}$ with
$h_{\black}(i_0) \ne h_{\black}(i_1)$.
For the same reason, we also have
$\inv_{\white,p_0,p_1} + \inv_{\white,p_1,p_0} =
b_0 b_1 + (b_0 + b_1) (b - b_0 - b_1)$ and we are done.
\end{proof}


\section{Computing the defect $\Delta$}
\label{section:Delta}

In this section we first give an estimate for $\Delta(\cR_N)$
as a function of $N$.
We then give an explicit formula in special cases.
In many cases the defect $\Delta(\cR)$ is easier to compute
than the number of tilings $|\cT(\cR)|$:
this is related to determinants being easier to compute
than permanents.

\begin{lemma}
\label{lemma:fu}
Let $\cD \subset \RR^{n-1}$.  
For every $p \in \cP$, if $p$ contains exactly $N$ unit cubes
then the cork $\cR_{0,N;\emptyplug,p}$ admits a tiling.
\end{lemma}


\begin{proof}
The proof is by induction on $b = N/2$, the number of black squares in $p$.
The case $b = 0$ is trivial.
For $b > 0$, consider a pair of unit cubes $v,w$ in $p$,
$v$ black, $w$ white, such that the distance between $v$ and $w$
(measured in $\cG_\cD$) is minimal.
Let $\tilde p = p \smallsetminus \{v,w\} \in \cP$. 
By the induction hypothesis there exists a tiling $\tilde\bt$ of
$\cR_{0,\tilde N;\emptyplug,\tilde p}$ for $\tilde N = N - 2$.
On the first $\tilde N$ floors $\bt$ coincides with $\tilde \bt$.
In order to construct the last two floors,
consider a path of minimal length from $v$ to $w$.
By minimality, this path intersects no other unit cubes in $p$.
For unit cubes not in the path the final two floors are filled
with a vertical domino.
Along the path we use horizontal dominoes, completing the construction.
\end{proof}


\begin{lemma}
\label{lemma:full}
Let $\cD \subset \RR^{n-1}$,
where $n \ge 4$.
Assume that $\cD$ contains a
$3\times 3 \times 1 \times \cdots \times 1$ box.
Then there exists $N_{\min}$ such that for all $p_0,p_1 \in \cP$
we have $|(A^N)_{p_0,p_1}| > |(\tilde A^N)_{p_0,p_1}|$
for all $N \ge N_{\min}$.
\end{lemma}

\begin{proof}
The vertical tiling $\bt_0$ of $\cR_2 = \cD \times [0,2]$
satisfies $\Tw(\bt_0) = 0$. 
Replace the vertical dominoes in the 
$3\times 3 \times 1 \times \cdots \times 1 \times 2$ box
by the dominoes shown in Figure \ref{fig:knot} to obtain
a tiling $\bt_1$ of $\cR_2 = \cD \times [0,2]$ with $\Tw(\bt_1) = 1$.
Let $b_{\cD}$ be the number of black unit cubes in $\cD$.
Apply Lemma \ref{lemma:fu} to the full plug $\fullplug$
to obtain a tiling $\bt_\bullet$ of the cork 
$\cR_{0,2b_{\cD};\emptyplug,\fullplug}$.  The tiling
$\bt_{\bullet}$ can be considered a tiling of $\cR_{2b_{\cD} - 1}$.
There are therefore tilings of $\cR_{N_0}$ of either twist
for $N_0 \ge 2b_{\cD} + 1$.

Take $N_{\min} = 6b_{\cD} + 1$ and $N \ge N_{\min}$.
Assume that $p_i$ contains $b_i$ black unit cubes.
Apply Lemma \ref{lemma:fu} to $p_0$ to obtain a tiling
of the cork $\cR_{0,2b_{0};p_0,\emptyplug}$.
Apply Lemma \ref{lemma:fu} to $p_1$ to obtain a tiling
of the cork $\cR_{N-2b_1,N;\emptyplug,p_1}$.
Clearly, $(N-2b_1) - 2b_0 \ge N_0$.
From the previous paragraph,
there exist tilings of $\cR_{2b_0, N-2b_1; \emptyplug, \emptyplug}$
of either twist.
Juxtapose the tilings to obtain two tilings
$\tilde\bt_0, \tilde\bt_1 \in \cT(\cR_{0,N;p_0,p_1})$
with contributions of opposite signs to $(\tilde A^N)_{p_0,p_1}$,
completing the proof.
\end{proof}

\begin{lemma}
\label{lemma:lambda}
Let $\cD \subset \RR^{n-1}$,
where $n \ge 4$.
Assume that $\cD$ contains a
$3\times 3 \times 1 \times \cdots \times 1$ box.
Then there exist $\lambda > 1$ and $C > 0$ such that
\[ \lim_{N \to \infty} \frac{|\cT(\cR_N)| - C\lambda^N}{\lambda^N} = 0. \]
Furthermore, $\Delta(\cR_N)$ as a function of $N$
is either eventually constant,
eventually periodic with period $2$ or
there exist $\tilde\lambda \in (1,\lambda)$ and $\tilde C > 0$ such that
\[ \limsup_{N \to \infty} \frac{|\Delta(N)|}{\tilde\lambda^N} = \tilde C. \]
\end{lemma}

\begin{proof}
We apply the Perron-Frobenius Theorem to the matrix $A$.
It follows from Lemma \ref{lemma:full} that there exists
a positive eigenvalue $\lambda$ such that all other eigenvalues
have strictly smaller absolute value.
The associated eigenvector has positive entries and therefore
$|\cT(\cR_N)| = (A^N)_{\emptyplug,\emptyplug}$ has a leading term
$C \lambda^N$:
any other term is exponentially smaller.
Since $A$ is symmetric and all entries of $A$ are integers,
all eigenvalues are real algebraic integers.
If an eigenvalue belongs to $\RR \smallsetminus \{-1,0,1\}$,
one of its conjugates must have absolute value larger than $1$
and therefore $\lambda > 1$.
If all eigenvalues belong to $\{-1,0,1\}$,
$|\cT(\cR_N)|$ is eventually periodic with period $1$ or $2$
and therefore bounded,
contradicting the construction in the proof of Lemma \ref{lemma:full}.

Again from Perron-Frobenius, all eigenvalues of $\tilde A$
have absolute value smaller than $\lambda$.
Let $\tilde\lambda$ be the maximum absolute value of eigenvalues of $\tilde A$.
If $\tilde\lambda \le 1$ then all eigenvalues belong to $\{-1,0,1\}$
and therefore 
$\Delta(\cR_N)$ is eventually periodic with period $1$ or $2$.
If $\tilde\lambda > 1$ then at least one of $\pm\tilde\lambda$
is an eigenvalue, implying the last estimate in the statement.
\end{proof}

\begin{coro}
\label{coro:quasibalance}
Let $\cD \subset \RR^{n-1}$,
where $n \ge 4$.
Assume that $\cD$ contains a
$3\times 3 \times 1 \times \cdots \times 1$ box.
Then
\[ \lim_{N \to \infty}
\frac{|\{\bt \in \cT(\cR_N) \;|\; \Tw(\bt) = 0 \}|}{|\cT(\cR_N)|} = \frac12,
\quad
\lim_{N \to \infty}
\frac{|\{\bt \in \cT(\cR_N) \;|\; \Tw(\bt) = 1 \}|}{|\cT(\cR_N)|} = \frac12.
\]
\end{coro}

\begin{proof}
From Lemma \ref{lemma:lambda} we have
\[ \lim_{N \to \infty} \frac{\Delta(\cR_N)}{|\cT(\cR_N)|} = 0. \]
The result follows from Equation \ref{equation:tw}.
\end{proof}

The following result, in a similar spirit,
will be needed to prove Corollary \ref{coro:twin}.

\begin{lemma}
\label{lemma:verts}
Let $\cD \subset \RR^{n-1}$,
where $n \ge 4$.
Assume that $\cD$ contains a
$3\times 3 \times 1 \times \cdots \times 1$ box.
There exists $c < 1$ with the following properties.
Let $M$ be a fixed positive integer.
Let $C_N$ be the number of tilings $\bt$ of $\cR_N$
with fewer than $M$ vertical floors.
Then
\[ \lim_{N \to \infty} \frac{C_N}{c^N |\cT(\cR_N)|} = 0. \]
\end{lemma}

\begin{proof}
Let $A$ be the adjacency matrix, as above.
Let $A_{\sharp}$ be the corresponding matrix,
but not counting vertical floors.
We have $|(A_{\sharp})_{i,j}| \le A_{i,j}$ for all $i,j \in \cP$,
with strict inequality for some entries.

Let $\lambda > 0$ be the eigenvalue of $A$ of largest absolute value,
as above. There exists $c < 1$ such that all eigenvalues of $A_{\sharp}$
have absolute value strictly smaller than $c\lambda$.
This is our desired $c$.

Consider an auxiliary $c_{-}$, $c_{-} < c$,
such that all eigenvalues of $A_{\sharp}$
also have absolute value smaller that $c_{-}\lambda$.
Thus, for any $i,j \in \cP$, 
\[ \lim_{N\to\infty}
\frac{(A_{\sharp}^N)_{i,j}}{c_{-}^N\lambda^N} = 0. \]
Thus, there exists a constant $C_{\sharp}$ such that
$|(A_{\sharp}^N)_{i,j}| < C_{\sharp} c_{-}^N\lambda^N$
for all $i,j \in \cP$ and all $N$.

We need an estimate for $C_N$.
The number of $M$-tuples $0 < k_1 < \cdots < k_M < N$
is bounded by $N^M$.
For each such $M$-tuple $(k_1, \ldots, k_M)$,
we count the number of tilings of $\cR_N$
where vertical floors are allowed only in the positions $k_i$.
We first choose floors and plugs in the positions $k_i$ and $k_{i+1}$:
there are a fixed number $K$ of such choices.
We then choose the tiling in each interval:
the initial and final plugs $p_i$ and $p_{i+1}$ are now fixed.
There are
$(A_{\sharp}^{k_{i+1}-k_i})_{p_i,p_{i+1}} <
C_{\sharp} c_{-}^{(k_{i+1}-k_i)}\lambda^{(k_{i+1}-k_i)}$
such tilings.
Thus, $C_N < N^M C_{\sharp} c_{-}^{N}\lambda^{N}$;
for large $N$, $C_N \ll c^{N}\lambda^{N}$, as desired.
\end{proof}


\section{The domino complex $\cC_\cD$}
\label{section:complex}

In this section we complete the construction of
the $2$-complex $\cC_\cD$.
The construction is very similar to the one performed in \cite{regulardisk},
thus we skip some details.
Recall from Section~\ref{section:plug} that tilings of $\cR_N = \cD \times [0,N]$
correspond to walks of length $N$ in $\cC_\cD$
from $\emptyplug \in \cP$ to $\emptyplug$.
We shall now see that $\bt_0 \sim \bt_1$
if and only if the corresponding continuous paths
are homotopic with fixed endpoints.

Consider the graph $\cC_{1,\cD}$ as a $1$-complex.
To  each self loop (always from $\emptyplug$ to itself)
attach the boundary of a M\"obius band.
Otherwise, we attach boundaries of $2$-cells (disks).
The $2$ cells correspond to flips 
as described below:

\begin{figure}[ht]
\centering
\def\svgwidth{120mm}
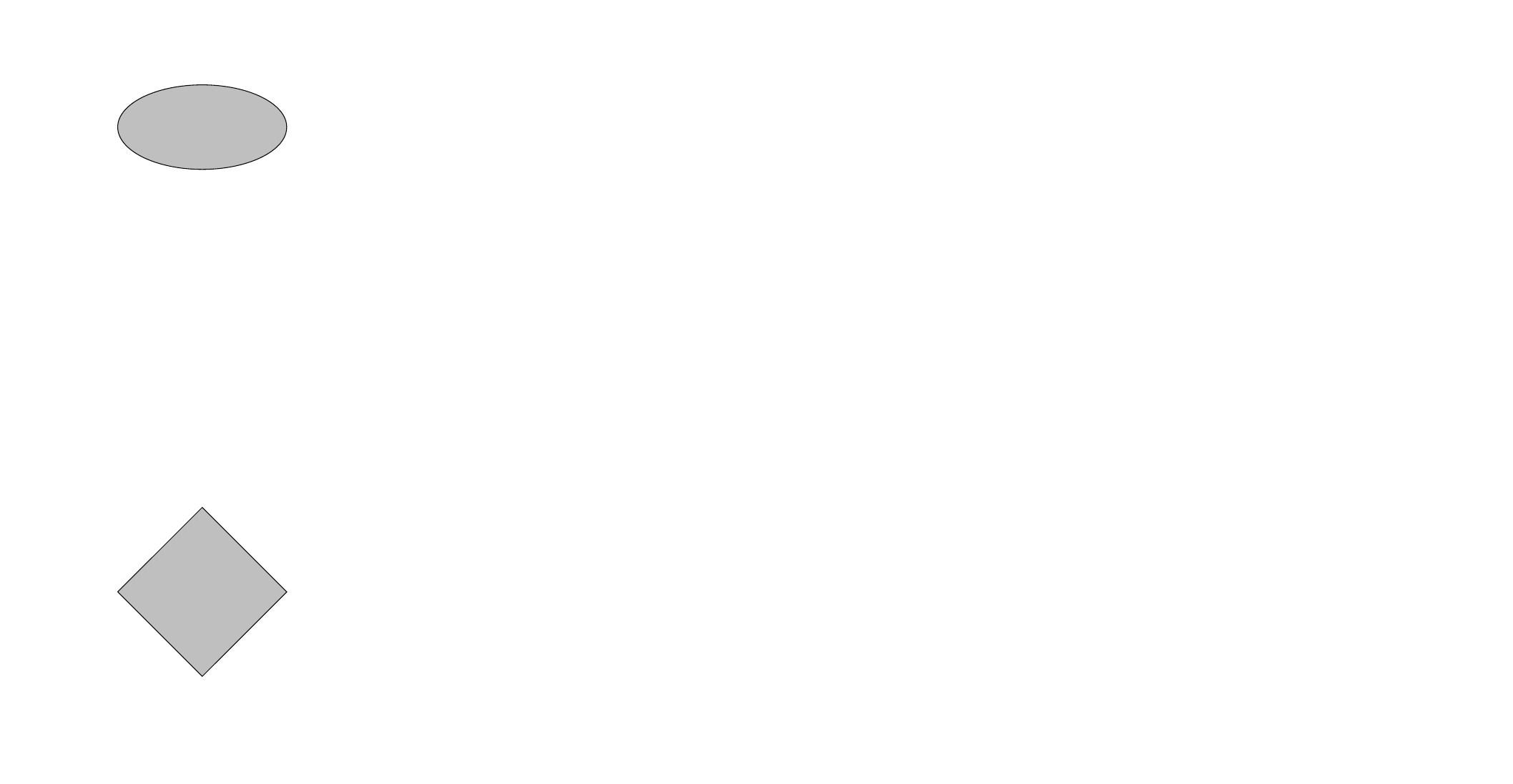
\caption{A flip manifests itself in the complex $\cD$ as a $2$-cell.
The figure shows a horizontal and a vertical flip.}
\label{fig:hvflip}
\end{figure}

First consider horizontal flips.
These join two tilings $\bt_0, \bt_1$ of $\cD_{p_0,p_1}$
(where $p_0, p_1 \in \cP$ are disjoint plugs).
In the complex, $p_0$ and $p_1$ are vertices
and $\bt_0$ and $\bt_1$ are $1$-cells joining them
(in other words, $\bt_0$ and $\bt_1$ are floors).
Attach to the complex a $2$-cell whose oriented boundary
is $\bt_0$ (from $p_0$ to $p_1$) 
followed by $\bt_1$ (from $p_1$ to $p_0$).
Combinatorially, the $2$-cell is a bigon.
Figure \ref{fig:hvflip} shows an example of such a $2$-cell.

Next consider vertical flips.
There are now two floors in play.
In one tiling, we have
$p_0, \bt_0, p_1, \bt_1, p_2$,
where $\bt_0 \in \cT(\cD_{p_0,p_1})$
and $\bt_1 \in \cT(\cD_{p_1,p_2})$.
In the  other we have
$p_0, \tilde\bt_0, \tilde p_1, \tilde\bt_1, p_2$,
where $\tilde\bt_0 \in \cT(\cD_{p_0,\tilde p_1})$
and $\tilde\bt_1 \in \cT(\cD_{\tilde p_1,p_2})$.
The plug $\tilde p_1$ is obtained from $p_1$
by removing two adjacent unit cubes.
These two unit cubes form dominoes
in both $\tilde\bt_0$ and $\tilde\bt_1$.
Again, attach to the complex a $2$-cell whose oriented boundary
is $\bt_0$ (from $p_0$ to $p_1$),
$\bt_1$ (from $p_1$ to $p_2$),
$\tilde\bt_1$ (from $p_2$ to $\tilde p_1$) and
$\tilde\bt_0$ (from $\tilde p_1$ to $p_0$).
Combinatorially, the $2$-cell is a square.
Figure \ref{fig:hvflip} also shows an example of
this other kind of $2$-cell.

By construction, if two tilings $\bt_0, \bt_1$ of $\cR_{0,N;p_0,p_N}$
differ by a flip, the two corresponding (continuous) paths are homotopic.
Indeed, we added a $2$-cell which guarantees just that.
Also, if a tiling $\bt_1$ of $\cR_{0,N+2;p_0,p_N}$
is obtained from a tiling $\bt_0$ of $\cR_{0,N;p_0,p_N}$
by inserting two vertical floors (at any position),
the two paths are trivially homotopic.
The converse statement is similar.
Thus, $G_{\cD}$ is naturally identified with
the fundamental group $\pi_1(\cC_{\cD},\emptyplug)$.

There is a natural surjective map $G_{\cD} \to \ZZ/(2)$
taking a tiling of $\cR_{0,N;p_0,p_N}$ to $N \bmod 2$.
The kernel of this map is $G_{\cD}^{+} < G_{\cD}$,
a normal subgroup of index $2$.

Since the complex is finite,
the group $G_{\cD}$ is finitely presented.
The immediate construction is far too complicated, however.
Later we shall significantly improve this situation.


\section{Hamiltonian regions and generators of $G_{\cD}$}
\label{section:hamiltonian}

The results from this section will be used repeatedly
to prove regularity of regions, or, more generally, to compute the domino group.
We recall the following fact.

\begin{fact}
\label{fact:planar}
Let $\cR \subset \RR^2$ be a planar balanced quadriculated region.
Let $\bt_0, \bt_1$ be tilings of $\cR$.
Then $\bt_0 \approx \bt_1$ if and only if $\Flux(\bt_0) = \Flux(\bt_1)$.
\end{fact}

For the proof of Fact \ref{fact:planar},
see \cite{thurston1990, saldanhatomei1995}.
The general concept of flux will not be required;
we will clarify the meaning in special cases when it comes up.

A cubiculated region $\cD \subset \RR^{n-1}$
is {\em Hamiltonian} if the graph $\cG_{\cD}$ admits a Hamiltonian path.
A fixed Hamiltonian path $\gamma_0 = (s_1, \ldots, s_M)$
is usually assumed; here $M = |\cD|$ and the $s_i$ are unit cubes.

\begin{example}
\label{example:hamiltonianbox}
Any box $\cD = [0,L_1] \times \cdots \times [0,L_{n-1}]$
is Hamiltonian.
We construct an explicit path recursively on $n$.
For $n = 2$, the path in $[0,L_1]$ is given by
$s_i = [i-1,i]$.
Assume a path $\gamma_0 = (s_1, \ldots, s_M)$ given in
$[0,L_1] \times \cdots \times [0,L_{n-1}]$,
where $M = L_1\cdots L_{n-1}$.
We construct a path $\tilde\gamma_0$ in
$[0,L_1] \times \cdots \times [0,L_{n}]$.
The number of unit cubes in the new box is $\tilde M = ML_n$.
For $\tilde k \in \ZZ$, $1 \le \tilde k \le \tilde M$,
let $x_n = \lceil\tilde k/M \rceil$
and $k = \tilde k - (x_n - 1)M$ if $x_n$ is odd
and $k = 1 + x_n M - \tilde k$ if $x_n$ is even.
Define $\tilde s_{\tilde k} = s_k \times [x_n - 1, x_n]$.
The next example is a special case.
\end{example}

\begin{example}
\label{example:22Lt}
Some small examples deserve special attention,
particularly $\cD = [0,2]^2 \times [0,L]$, $L \ge 2$.
The construction from Example \ref{example:hamiltonianbox} applies,
but a variation is easier to draw.

Consider the quadriculated cylinder
$\hat\cD = (\RR/(4\ZZ)) \times [0,L]$:
the bipartite graphs $\cG_{\cD}$ and $\cG_{\hat\cD}$ are isomorphic.
It follows that
the bipartite graphs $\cG_{\cR_N}$ and
$\cG_{\hat\cR_N}$
are also isomorphic (for any $N \in \NN^\ast$),
where $\hat\cR_N = \hat\cD \times [0,N]$
is a 3D cubiculated manifold.

Tilings of $\hat\cR_N$ can be represented as in Figure \ref{fig:hatRN}.
Here, floors are shown sequentially.
The quadriculated cylinder $\hat\cD$ is represented
by a rectangle where the right and left sides are identified
(as in a Mercator map).
Each floor $f_i$ is of the form $f_i = (p_{i-1},f_i^{\ast},p_i)$.
Here, as in Equation \ref{equation:plugsandfloors},
$p_{i-1}, p_i \in \cP_{\hat\cD}$ are disjoint plugs
and $f_i^{\ast}$ is a tiling of $\hat\cD_{p_{i-1},p_1}$.

\begin{figure}[ht]
\begin{center}
\includegraphics[scale=0.275]{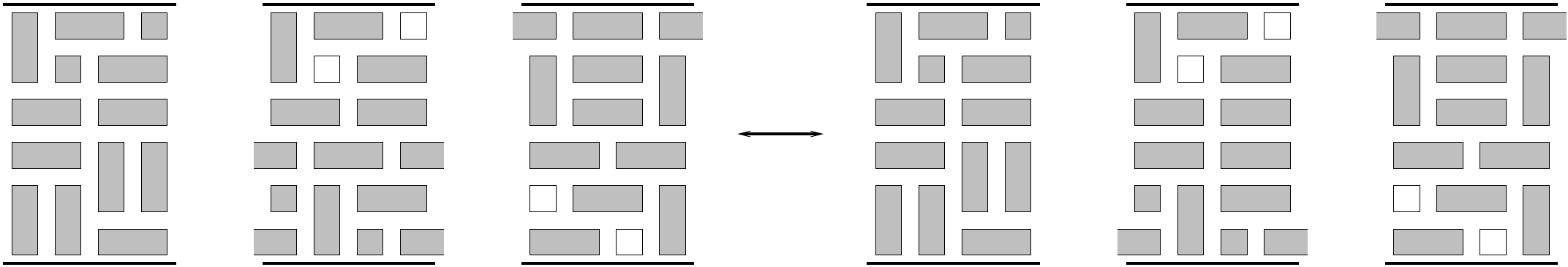}
\end{center}
\caption{Two tilings of $\hat\cR_N$ for $N = 3$ and $L = 6$.
The two tilings differ by a pseudoflip in the fourth row
of the second floor.}
\label{fig:hatRN}
\end{figure}

There exists an important difference between $\cR_N$ and $\hat\cR_N$,
however.  Some flips in $\cR_N$ are represented in $\hat\cR_N$ not in
the usual way, but as {\em pseudoflips}.  In a pseudoflip, a row (of
length $4$) of $\hat\cD$ is rotated by one unit, as shown in Figure
\ref{fig:hatRN}. 
In $\cR_N$, which has higher dimension,
a pseudoflip is an honest flip.

The regions $\cD = [0,2]^2 \times [0,L]$
and $\hat\cD = C_4 \times [0,L]$ are Hamiltonian.
Figure \ref{fig:22Lt} shows Hamiltonian paths in $\tilde\cD$
for $L = 3, 4, 5$.
\end{example}

Recall that a domino is {\em horizontal} if it is contained
in a single floor and is {\em vertical} otherwise.
We say that a horizontal domino
{\em respects the path} if and only if
it corresponds to an edge along the path;
vertical dominoes always respect the path.
A tiling {\em respects the path} if and only if it consists only
of dominoes which respect the path.

\begin{figure}[ht]
\begin{center}
\includegraphics[scale=0.275]{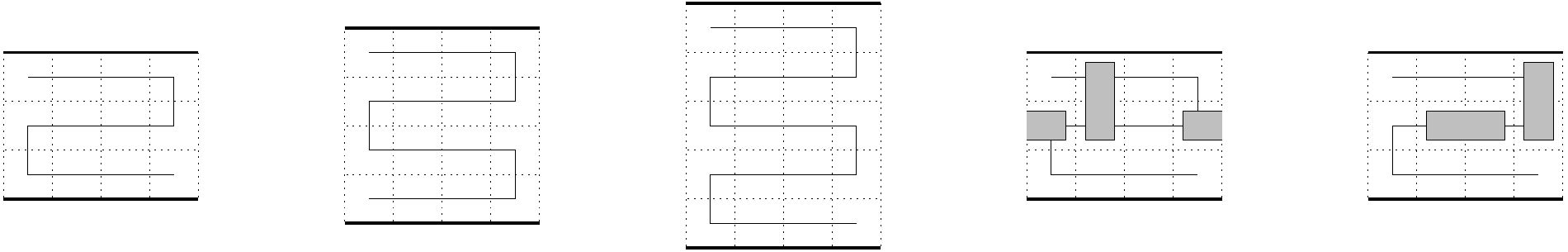}
\end{center}
\caption{The first three diagrams show Hamiltonian paths $\gamma_0$
in the quadriculated surfaces $\hat\cD = C_4 \times [0,L]$
for $L = 3, 4, 5$.
The fourth diagram shows two horizontal dominoes
which do not respect the path.
The fifth diagram shows two horizontal dominoes
which respect the path.}
\label{fig:22Lt}
\end{figure}

Consider a domino $d \subset \cD$
which does not respect the path.
We have $d = s_{i_{d,-}} \cup s_{i_{d,+}} \subset \cD$
with $i_{d,-} + 1 < i_{d,+}$.
The domino $d$ decomposes the path $\gamma_0$ into intervals
(finite sets of integers):
\begin{gather*}
I_{d;-} = \ZZ \cap [1,i_{d,-}-1], \quad 
I_{d;+} = \ZZ \cap [i_{d,+}+1,M], \\
I_{d;0} = \ZZ \cap [i_{d,-}+1,i_{d,+}-1].
\end{gather*}
The set $I_{d;0}$ always has even and positive cardinality.
A plug $p \in \cP$ is {\em compatible} with $d$
if and only if it does not include a square of $d$.
If $p$ is a plug compatible with $d$ define:
\[ \flux(d;p) = (\flux_{-}(d;p),\flux_{0}(d;p),\flux_{+}(d;p)), \quad
\flux_j(d;p) = \sum_{i \in I_{d;j}, s_i \subset p} (-1)^i. \]
Let $H \subset \ZZ^3$ be the perpendicular lattice to $(1,1,1)$;
let $\Phi_d \subset H$ be the finite set of values of $\flux(d;p)$
for $p \in \cP$ compatible with $d$.

\begin{lemma}
\label{lemma:tdp}
Consider a Hamiltonian region $\cD \subset \RR^{n-1}$
with a fixed path $\gamma_0$.
\begin{enumerate}
\item{Let $d \subset \cD$ be a domino which does not respect the path.
Let $p \in \cP$ be a plug compatible with $d$.
Then, for sufficiently large even $N$ there exists
a tiling $\bt_{d;p}$ of $\cR_{2N}$ with the following properties.
There exists a unique domino in $\bt_{d;p}$ which does not respect the path:
$d \times [N-1,N]$.
The plug of $\bt_{d;p}$ at height $N-1$ is $p$.}
\item{Let $d \subset \cD$ be a domino which does not respect the path.
Let $p_0, p_1 \in \cP$ be plugs compatible with $d$.
Let $\bt_{d;p_0}, \bt_{d;p_1}$ be tilings of $\cR_{2N}$
satisfying the conditions of the first item.
If $\flux(d;p_0) = \flux(d;p_1)$ then $\bt_{d;p_0} \approx \bt_{d;p_1}$.
Also, the sequence of flips from $\bt_{d;p_0}$ to $\bt_{d;p_1}$
can be chosen so as to keep the domino $d \times [N-1,N]$ fixed
and all other dominoes respect the path.}
\end{enumerate}
\end{lemma}

\begin{proof}
A tiling of $\cR_{2N}$ which respects the path
can be {\em unfolded} to obtain a tiling of
$\tilde\cR_{2N} = [0,M] \times [0,2N]$.
A horizontal domino in $\cR_{2N}$ of the form
$\tilde d \times [j-1,j]$, $\tilde d = s_i \cup s_{i+1} \subset \cD$,
is taken to $[i-1,i+1] \times [j-1,j] \subset \tilde\cR_{2N}$.
A vertical domino in $\cR_{2N}$ of the form
$s_i \times [j-1,j+1]$ 
is taken to $[i-1,i] \times [j-1,j+1] \subset \tilde\cR_N$.

Similarly, consider a tiling $\bt$ of $\cR_{2N}$ such that
there exists a unique domino in $\bt_{d;p}$ which does not respect the path:
$d \times [N-1,N]$.
The tiling $\bt$ can be unfolded to obtain a tiling $\tilde\bt$
of the planar region $\tilde\cR_{2N,d}$:
\begin{gather*}
\tilde\cR_{2N,d} = ([0,M] \times [0,2N]) 
\smallsetminus (s_{-} \cup s_{+}) \subset \RR^2, \\
s_{-} = [i_{d,-}-1,i_{d,-}] \times [N-1,N], \qquad
s_{+} = [i_{d,+}-1,i_{d,+}] \times [N-1,N]. 
\end{gather*}
Conversely, a tiling $\tilde\bt$ of $\tilde\cR_{2N,d}$
can be folded to obtain a tiling $\bt$ of $\cR_{2N}$
with the properties above.

For the first item, the information about plugs
reduces the problem to tiling two similar contractible planar regions.
The first region is obtained from the rectangle $[0,M] \times [0,N-1]$
by removing from row $[0,M] \times [N-2,N-1]$
the unit squares contained in $p$.
The second region is obtained from the rectangle $[0,M] \times [N-1,2N]$
by removing from row $[0,M] \times [N-1,N]$
both the unit squares contained in $p$ and 
the domino $d$.
This is discussed in \cite{regulardisk};
see also \cite{thurston1990}.

For the second item, unfold the tilings $\bt_0$ and $\bt_1$
to obtain tilings $\tilde\bt_0$ and $\tilde\bt_1$ of
the planar region $\tilde\cR_{2N,d}$.
The condition $\flux(d;p_0) = \flux(d;p_1)$
is translated to $\Flux(\tilde\bt_0) = \Flux(\tilde\bt_1)$.
From Fact \ref{fact:planar}, $\tilde\bt_0 \approx \tilde\bt_1$.
Take the sequence of flips for the planar problem
and fold back to obtain the desired sequence of flips
in $\cR_{2N}$.
\end{proof}

Consider a domino $d \subset \cD$ not respecting the path
and $\phi \in \Phi_d \subset H \subset \ZZ^3$.
Choose $p \in \cP$, $p$ compatible with $d$, $\flux(d;p) = \phi$.
Apply the first item of Lemma \ref{lemma:tdp}
to obtain a tiling $\bt_{d;\phi} = \bt_{d;p}$
with the properties listed in that item.
Notice that the second item implies that,
for fixed $N$ (but independently of $p$),
all such tilings are mutually connected by sequences of flips.

\begin{lemma}
\label{lemma:generators}
The family of tilings
$(\bt_{d;\phi})$, $d \subset \cD$ not respecting the path $\gamma_0$,
$\phi \in \Phi_d$,
generates the subgroup $G^{+}_{\cD} < G_{\cD}$.
\end{lemma}

\begin{proof}
Recall that $G^{+}_{\cD} < G_{\cD}$ is a normal subgroup of index $2$,
the kernel of the natural surjective map $G_{\cD} \to \ZZ/(2)$
(parity of length of walk).
The proof follows with very slight adaptations the proof of
Corollary 8.6 in \cite{regulardisk}.
\end{proof}

\section{Irregularity of $\cD = [0,2]^3$}
\label{section:small}

We now discuss the smallest non trivial example:
see Example \ref{example:222}.

\begin{lemma}
\label{lemma:222}
Let $\cD = [0,2]^3$.
There exists a surjective map $\Tw_{\ZZ}: G_{\cD} \to \ZZ$
such that $\Tw(\bt) = \Tw_{\ZZ}(\bt) \bmod 2$ 
for any tiling $\bt$ of $\cD \times [0,N]$, $N \in \NN^\ast$.
In particular, $\cD$ is not regular.
\end{lemma}

\begin{proof}

Consider a domino $d$ and a square $s$ contained in
$\hat\cD =  C_4 \times [0,2]$:
we define $\tau(d,s) \in \{ - \frac14, 0, \frac14 \}$
as in Figure \ref{fig:tauhat}.
For other configurations, $\tau(d,s) = 0$.
Thus, $\tau(d,s) \ne 0$ if and only if
$d$ and $s$ are disjoint,
$d \subset \hat\cD$ is in the $C_4$ direction
and a projection onto $C_4$
takes $s$ to a subset of $d$.

\begin{figure}[ht]
\centering
\def\svgwidth{70mm}
\begingroup%
  \makeatletter%
  \providecommand\color[2][]{%
    \errmessage{(Inkscape) Color is used for the text in Inkscape, but the package 'color.sty' is not loaded}%
    \renewcommand\color[2][]{}%
  }%
  \providecommand\transparent[1]{%
    \errmessage{(Inkscape) Transparency is used (non-zero) for the text in Inkscape, but the package 'transparent.sty' is not loaded}%
    \renewcommand\transparent[1]{}%
  }%
  \providecommand\rotatebox[2]{#2}%
  \newcommand*\fsize{\dimexpr\f@size pt\relax}%
  \newcommand*\lineheight[1]{\fontsize{\fsize}{#1\fsize}\selectfont}%
  \ifx\svgwidth\undefined%
    \setlength{\unitlength}{681.03807338bp}%
    \ifx\svgscale\undefined%
      \relax%
    \else%
      \setlength{\unitlength}{\unitlength * \real{\svgscale}}%
    \fi%
  \else%
    \setlength{\unitlength}{\svgwidth}%
  \fi%
  \global\let\svgwidth\undefined%
  \global\let\svgscale\undefined%
  \makeatother%
  \begin{picture}(1,0.23152158)%
    \lineheight{1}%
    \setlength\tabcolsep{0pt}%
    \put(0,0){\includegraphics[width=\unitlength,page=1]{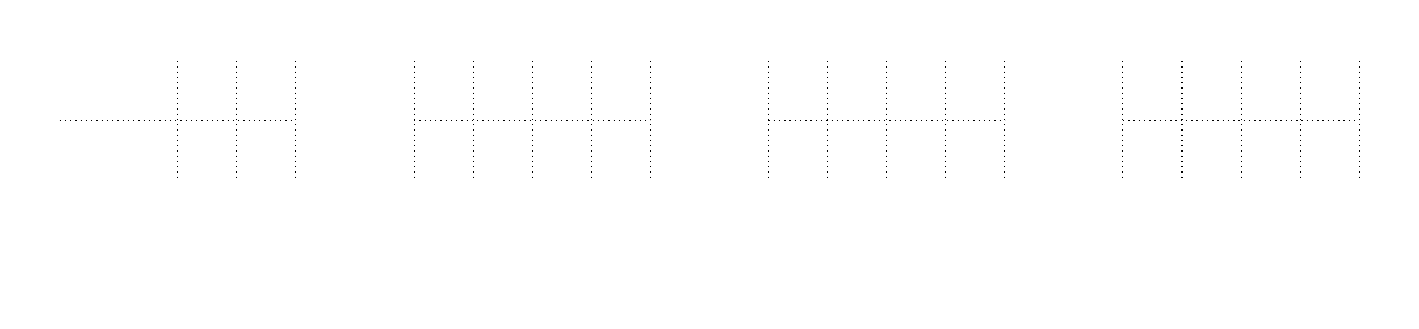}}%
    \put(0.1044481,0.0427816){\color[rgb]{0,0,0}\makebox(0,0)[lt]{\lineheight{1.25}\smash{\begin{tabular}[t]{l}$+1$\end{tabular}}}}%
    \put(0.60409261,0.0427816){\color[rgb]{0,0,0}\makebox(0,0)[lt]{\lineheight{1.25}\smash{\begin{tabular}[t]{l}$-1$\end{tabular}}}}%
    \put(0.35427035,0.0427816){\color[rgb]{0,0,0}\makebox(0,0)[lt]{\lineheight{1.25}\smash{\begin{tabular}[t]{l}$-1$\end{tabular}}}}%
    \put(0.85391486,0.0427816){\color[rgb]{0,0,0}\makebox(0,0)[lt]{\lineheight{1.25}\smash{\begin{tabular}[t]{l}$+1$\end{tabular}}}}%
    \put(0,0){\includegraphics[width=\unitlength,page=2]{tauhat.pdf}}%
  \end{picture}%
\endgroup%

\caption{The value of $4\tau(d,s)$ in four examples.
The sign depends on two bits:
the horizontal position of the square and the relative position of the square and domino.  We can give signs in a consistent way only in this small case.}
\label{fig:tauhat}
\end{figure}

Recall that a plug $p \in \cP_{\hat\cD}$
is a balanced subset of $\hat\cD$.
If $p, \tilde p \in \cP_{\hat\cD}$ are disjoint plugs
then $\hat\cD_{p, \tilde p} =
\hat\cD \smallsetminus (p \cup \tilde p)$.
For disjoint plugs $p, \tilde p \in \cP_{\hat\cD}$
and $f \in \cT(\hat\cD_{p, \tilde p})$ define
\begin{equation}
\label{equation:cocycle}
\tau(f,p) = \sum_{d \in f, s \in p} \tau(d,s)
\in \frac14\ZZ; \qquad
\tau(f; p, \tilde p) = \tau(f,\tilde p) - \tau(f,p) \in \frac14 \ZZ.
\end{equation}
Draw a tiling $\bt \in \cT(\cR_N)$ as a sequence of floors,
as in Figure \ref{fig:hatRN}.
A tiling is therefore an alternating sequence of plugs and floors,
\[ \bt = (p_0, \ldots,f_i,p_i,f_{i+1},p_{i+1},\ldots, p_N), \]
with $p_i \in \cP_{\hat\cD}$ (for all $i$),
$p_0 = p_N = \emptyplug$ and
$f_i \in \cT(\hat\cD_{p_{i-1},p_i})$.
Define
\[ \Tw_{\ZZ}(\bt) = \sum_{0 < j \le N}
\tau^u(\floorop_j(\bt); \plug_{j-1}(\bt), \plug_j(\bt)). \] 
It is now not hard to verify that
$\Tw_{\ZZ}(\bt) \in \ZZ$ for any tiling $\bt$
and that $\Tw_{\ZZ}(\bt)$ is invariant under flips and pseudoflips.
\end{proof}



\begin{remark}
\label{remark:2222isolated}
Recall from Example \ref{example:222} that the box $[0,2]^4$
admits $272$ tilings, among them $8$ which are isolated, i.e.,
admit no flip. Figure \ref{fig:2222isolated} shows an example;
the others are obtained by rotation and reflection.
See also the second tiling in Figure \ref{fig:2222}.
With the concept of twist as $\Tw_{\ZZ}$, defined in Lemma \ref{lemma:222},
four of the $8$ isolated tilings have twist $+1$ and four have twist $-1$.

\begin{figure}[ht]
\begin{center}
\includegraphics[scale=0.275]{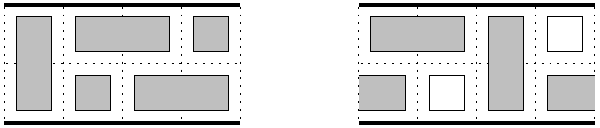}
\end{center}
\caption{An isolated tiling of the box $[0,2]^4$,
represented here as $\hat\cR_2$.}
\label{fig:2222isolated}
\end{figure}

Let $\bt_i$, $i \in \{1,2,3,4\}$, be the four isolated tilings with twist $+1$.
The tilings with twist $-1$ are the reflections $\bt_i^{-1}$.
Let $\bt_{\thin}$ be a tiling of $\cR_1$
(they are all $\approx$-equivalent).
Let $\bt_{\nvert}$ be the vertical tiling of $\cR_6$.
The brute force study of tilings of $\cR_6$ shows that,
for all $i \in \{1,2,3,4\}$,
\[ \bt_1 \ast \bt_{\thin} \ast \bt_{\thin} \ast \bt_i^{-1} \approx
\bt_1 \ast \bt_{\thin} \ast \bt_i^{-1} \ast \bt_{\thin} \approx \bt_{\nvert}.
\]
This implies that
the four tilings $\bt_i$ represent the same element
of the domino group $G_{\cD}$.
Also, $\bt_1 \ast \bt_{\thin} \ast \bt_{\thin} \ast \bt_i$ 
has twist $2$ and belongs to a component of size $98144$.
\end{remark}

\begin{lemma}
\label{lemma:222group}
Let $\cD = [0,2]^3$.
The domino group $G_{\cD}$ is isomorphic to $\ZZ \oplus \ZZ/(2)$.
Any tiling of $[0,2]^4$ which admits no flips
(such as the one in Figure \ref{fig:2222isolated})
is a generator of the $\ZZ$ component.
Any tiling of $[0,2]^3 \times [0,1]$ is a generator
of the $\ZZ/(2)$ component.
\end{lemma}

\begin{proof}
Consider the homomorphism from $G_{\cD}$ to $\ZZ \oplus \ZZ/(2)$
taking $\bt \in \cT(\cR_N)$ to $(\Tw_{\ZZ}(\bt), N \bmod 2)$:
this homomorphism is clearly surjective.

Let $\bt_1$ and $\bt_{\thin}$ be as in Remark \ref{remark:2222isolated}.
Consider the homomorphism from $\ZZ \oplus \ZZ/(2)$ to $G_{\cD}$ 
taking $(u,v)$ to $\bt_1^u \ast \bt_{\thin}^v$.
This homomorphism is clearly injective,
all we have to do is show that it is surjective.

Consider the generators $\bt_{d;\phi}$ of $G^{+}_{\cD}$
constructed in Lemma \ref{lemma:generators}.
Computations show that each $\bt_{d;\phi}$
is homotopic to a power of $\bt_1$; i.e. $\bt_1 \ast \bt_1 \ast \cdots \ast \bt_1$.  
\end{proof}


\section{Regularity of $\cD = [0,2]^2 \times [0,L]$, $L \ge 3$}
\label{section:22L}

We now discuss other small examples,
the boxes $\cD = [0,2]^2 \times [0,L]$ for $L \ge 3$, 
as in Example \ref{example:223}.

\begin{lemma}
\label{lemma:22L}
The box $\cD = [0,2]^2 \times [0,L]$ is regular for $L \ge 3$.
\end{lemma}

We shall need the following facts.  Fact \ref{fact:4L} is a special case of
the first main theorem in \cite{regulardisk}.

\begin{fact}
\label{fact:4L}
The rectangle $\cD_0 = [0,4] \times [0,L]$ is regular for $L \ge 3$.
\end{fact}

\begin{fact}
\label{fact:2233}
Let $\cD = [0,2]^2 \times [0,3]$ and
$\cR_3 = \cD \times [0,3]$.
If $\bt_0$ and $\bt_1$ are tilings of $\cR_3$ with
$\Tw(\bt_0) = \Tw(\bt_1) = 1$ then $\bt_0 \approx \bt_1$.
\end{fact}

Fact \ref{fact:2233} can be verified by brute force.
As mentioned in Example \ref{example:223},
all tilings of $\cR_3$ of twist $1$ form a connected component of size $99280$.

\begin{proof}[Proof of Lemma \ref{lemma:22L}]
We follow the construction in the previous section,
particularly Lemma \ref{lemma:generators}.
We use the Hamiltonian paths in Example \ref{example:22Lt}.
Let $\bt_0 = \bt_{\nvert}$ be the vertical tiling of $\cR_2$.
Let $\bt_{\pm 1}$ be the tilings of $\cR_4$ shown in Figure \ref{fig:223j}
(for $L > 3$, the other rows are similar
to the third row in the figure).
These two tilings are of the form 
$\bt_{+1} = \bt_{d;\phi}$ and $\bt_{-1} = \bt_{d;\phi'}$ where
the domino $d$ is the only one which does not respect the path.
We have $\tw(\bt_{+1}) = \tw(\bt_{-1}) = 1 \in \ZZ/(2)$.
It follows from Fact \ref{fact:2233}
(indeed, from Remark \ref{remark:2233})
that $\bt_{+1} \approx \bt_{-1}$.
We prove that $\bt_{+1}$ generates $G^{+}_{\cD} \approx \ZZ/(2)$,
which implies regularity.

\begin{figure}[ht]
\begin{center}
\includegraphics[scale=0.275]{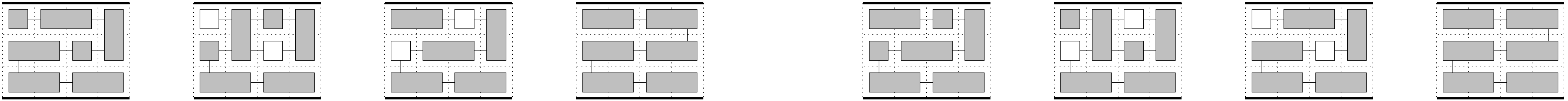}
\end{center}
\caption{Two tilings $\bt_{+1}$ and $\bt_{-1}$
of $\cD \times [0,4]$ for $\cD = [0,2]^2 \times [0,3]$.}
\label{fig:223j}
\end{figure}

Notice that, for given $L$,
this already reduces the proof to a finite
and reasonably small computation.
Indeed, for each domino $d$ not respecting the path
compute $\Phi_d \subset H$.
For each pair $(d,\phi)$, $\phi \in \Phi_d$,
construct a tiling  $\bt_{d;\phi}$.
Compute the twist of these tilings.
For each pair $(d,\phi)$, we must verify that
if $\tw(\bt_{d,\phi}) = s$ then $\bt_{d,\phi} \sim \bt_s$.

We now address the general case $L \ge 3$.
If the domino $d$ does not cross the sides of the rectangle,
the other dominoes will also not cross (they respect the path).
We may therefore consider $\bt_{d;\phi}$ to be a tiling of
$[0,4]\times[0,L]\times[0,N]$.
With this interpretation, we are fully in the scenario
of \cite{regulardisk}, and we know that $[0,4]\times[0,L]$ is regular:
this is Fact \ref{fact:4L}.
We stress that \textit{regularity} in the previous sentence
means \textit{regularity in the 3d sense}.
In other words, let $\cD_0$ be the quadriculated disk
$\cD_0 = [0,4] \times [0,L]$;
let $G_{\cD_0}$ be the domino group of $\cD_0$
(as defined in \cite{regulardisk}):
we have $G_{\cD_0} \approx \ZZ \oplus \ZZ/(2)$.
Thus, regularity of $\cD_0$ implies that $\bt_{d;\phi}$ is homotopic
(still with basis $\cD_0 = [0,4]\times[0,L]$)
to a product of copies of $\bt_{+1}$ and $\bt_{-1}$.
This in turn implies (now with basis $[0,2]^2 \times [0,L]$)
that $\bt_{d,\phi}$ is homotopic to a product of copies
of $\bt_{+1}$ and $\bt_{-1}$.
With basis $[0,2]^2 \times [0,L]$,
$\bt_{+1}$ and $\bt_{-1}$ are homotopic to each other
and both have degree $2$ (from Fact \ref{fact:2233}).
Thus, $\bt_{d,\phi}$ is homotopic to either $\bt_1$ or $\bt_0$.
In other words, $\bt_{d,\phi} \sim \bt_s$ where $s = \tw(\bt_{d,\phi})$.
We are then left with checking the $L$ horizontal dominoes
which cross the side of the rectangle.
Figure \ref{fig:223x} shows these dominoes for $L = 3$.

\begin{figure}[ht]
\begin{center}
\includegraphics[scale=0.275]{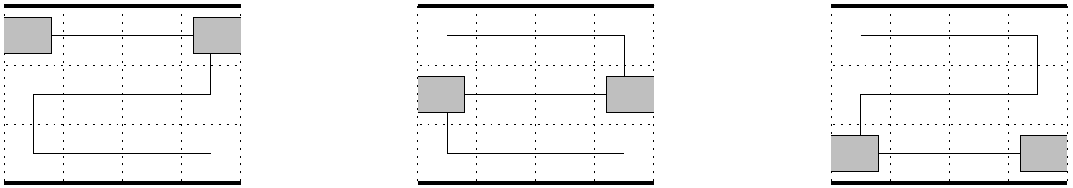}
\end{center}
\caption{Three dominoes which cross the side of the rectangle.}
\label{fig:223x}
\end{figure}

Notice that for these dominoes,
the central interval $I_{d;0}$ has exactly two elements
and therefore $|\phi_0| \le 1$, $\phi_0 = \flux_0(d;p)$.
We first address the case $\phi_0 = 0$.
In this case we may assume that the two elements of $I_{d;0}$ 
are covered by a domino in $\bt_{d;\phi}$
(we are using Fact \ref{fact:planar} here),
as in the first tiling of Figure \ref{fig:223y}.
A pseudoflip then takes $\bt_{d;\phi}$ to 
a tiling $\bt$ which everywhere respects the path.
We thus have $\bt_{d;\phi} \approx \bt \sim \bt_0$,
taking care of this case.

\begin{figure}[ht]
\begin{center}
\includegraphics[scale=0.275]{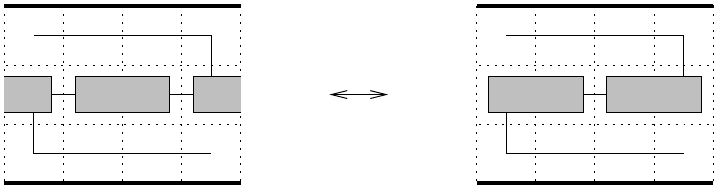}
\end{center}
\caption{The case $\phi_0 = 0$.}
\label{fig:223y}
\end{figure}

\begin{figure}[ht]
\begin{center}
\includegraphics[scale=0.275]{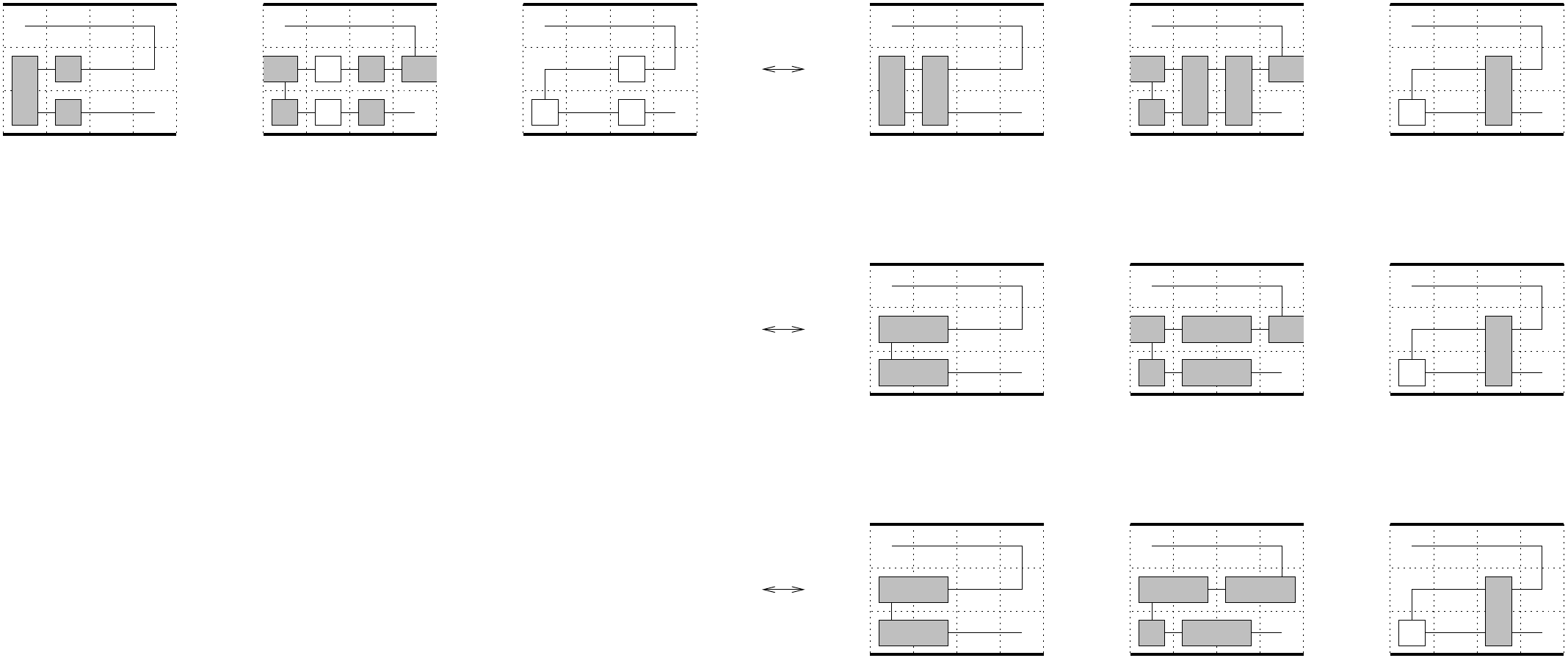}
\end{center}
\caption{The case $|\phi_0| = 1$.}
\label{fig:223z}
\end{figure}

\begin{figure}[ht]
\begin{center}
\includegraphics[scale=0.275]{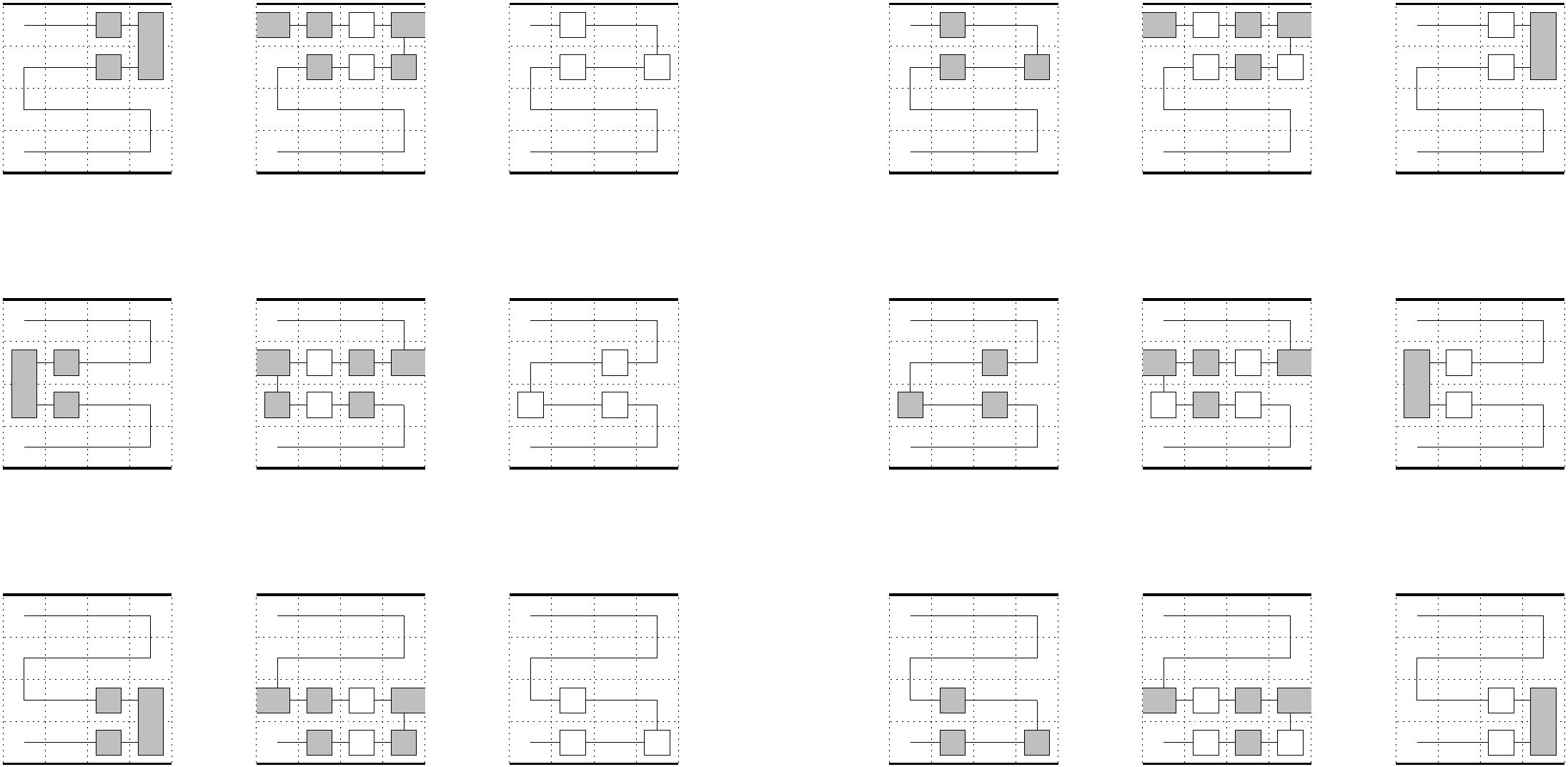}
\end{center}
\caption{The case $|\phi_0| = 1$, other configurations.}
\label{fig:224z}
\end{figure}

For the case $|\phi_0| = 1$, we may assume without loss of generality
that $\phi_0 \phi_{+} < 0$ (and $\gamma_0$ moves from top to bottom).
We may therefore assume that we have a configuration
similar to the one in Figure \ref{fig:223z} 
(other dominoes respecting the path are not shown);
different configurations are minor variations:
we show them for $L = 4$ in Figure \ref{fig:224z}.
A sequence of flips (and pseudoflips)
takes us to a tiling $\bt$ of $[0,4]\times[0,L]\times[0,N]$.
As in the previous case, $\bt$ is equivalent
to a product of copies of $\bt_{+1}$ and $\bt_{-1}$,
completing this last case and the proof of the lemma.
\end{proof}


\section{Regularity of boxes for $n = 4$}
\label{section:box4}

In this section we prove regularity for almost all boxes in $n = 4$.  
We shall need the following result which appears in \cite{regulardisk}.

\begin{fact}
\label{fact:regularrectangle}
Let $\tilde\cD = [0,L_a]\times[0,L_b] \subset \RR^2$ be a rectangle
with $L_a, L_b \in \NN^\ast$, $L_a, L_b \ge 3$, $L_aL_b$ even.
Then the rectangle $\tilde\cD$ is regular.
\end{fact}

Thus, if $N$ is even,
a tiling $\bt$ of $\tilde\cD \times [0,N]$ 
is homotopic (i.e., $\sim$-equivalent)
to a product of finitely many copies of 
the tilings $\bt_{+1}$ and $\bt_{-1}$ of $\tilde\cD \times [0,4]$,
shown in Figure \ref{fig:tpm} for $L_a = L_b = 4$.
In general, the upper $2 \times 3$ rectangle is as shown
and the rest is filled with dominoes respecting the path.

\begin{figure}[ht]
\begin{center}
\includegraphics[scale=0.275]{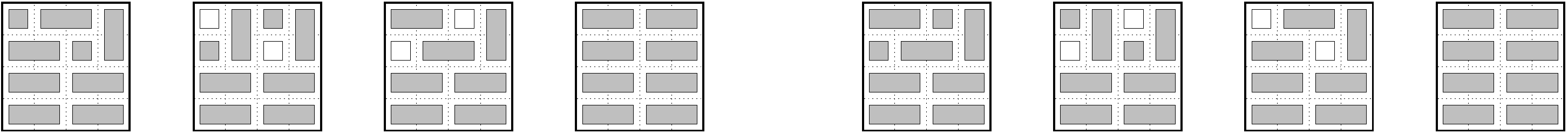}
\end{center}
\caption{The tilings $\bt_{\pm 1}$ for $L_a = L_b = 4$.}
\label{fig:tpm}
\end{figure}

Consider two quadriculated or cubiculated regions $\cR_1$, $\cR_2$
with Hamiltonian paths $\gamma_1$ and $\gamma_2$.
We say that $\cR_2$ is obtained by {\em folding} $\cR_1$
if and only if
$|\cR_1| = |\cR_2|$ and, for all $k_0, k_1$,
if $\gamma_1(k_0)$ and $\gamma_1(k_1)$ are adjacent then
$\gamma_2(k_0)$ and $\gamma_2(k_1)$ are also adjacent.
We also say that $\cR_1$ is obtained from $\cR_2$ by {\em unfolding}.
A trivial example is that a Hamiltonian region is obtained
by folding the path itself.
The following example will be used more than once.

\begin{example}
\label{example:foldbox}
The box $\cR_2$ below (of dimension $n$)
is obtained by folding $\cR_1$ (of dimension $n-1$):
\begin{align*}
\cR_1 &= [0,L_1] \times \cdots
\times [0,L_{k}L_{k+1}] \times \cdots \times [0,L_n],\\
\cR_2 &= [0,L_1] \times \cdots
\times [0,L_{k}] \times [0,L_{k+1}] \times \cdots \times [0,L_n],
\end{align*}
The Hamiltonian paths $\gamma_i$ are defined
as in Example \ref{example:hamiltonianbox}.
\end{example}

If $\cR_2$ is obtained from $\cR_1$ by folding,
any tiling of $\cR_1$ can be folded to define a tiling of $\cR_2$.
More precisely, if $\gamma_1(k_0)$ and $\gamma_1(k_1)$ form a domino
(contained in $\cR_1$) then
$\gamma_2(k_0)$ and $\gamma_2(k_1)$ also form a domino 
(contained in $\cR_2$).
The converse may be true or not:
a tiling of $\cR_2$ may or may not admit unfolding to $\cR_1$.

\begin{theo}
\label{lemma:box4}
Consider a cubiculated box $\cD = [0,L_1] \times [0,L_2] \times [0,L_3]$,
$L_i \ge 2$, $L_1L_2L_3$ even.
The box $\cD$ is regular {\em except} if $L_1 = L_2 = L_3 = 2$.
\end{theo}

\begin{proof}
The cases where at least two of the $L_i$ equal $2$
have already been discussed.
We may therefore assume $L_1, L_3 \ge 3$.

Consider the rectangles $\cD_1 = [0,L_1L_2] \times [0,L_3]$,
with Hamiltonian path $\gamma_1$,
and $\cD_2 = [0,L_1] \times [0,L_2L_3]$,
with Hamiltonian path $\gamma_2$.
As we saw in Example \ref{example:foldbox},
there is a {\em folding} procedure
from $\cD_i$ to $\cD$
and an {\em unfolding} procedure from $\cD$ to $\cD_i$
(for $i \in \{1,2\}$).

A tiling of $\cD_i \times [0,N]$ (for $i \in \{1,2\}$)
can always be folded to obtain a tiling of $\cD \times [0,N]$.
Both $\cD_1$ and $\cD_2$ are rectangles
satisfying the conditions of Fact \ref{fact:regularrectangle}.
We therefore have tilings $\bt_{\pm 1}$
of $\cD_i \times [0,4]$,
contructed as in Figure \ref{fig:tpm}
and satisfying $\Tw_{\ZZ}(\bt_{\pm 1}) = \pm 1$
(notice that $\cD_i \times [0,4]$ has dimension $3$,
so that we are in the situation where twist assumes values in $\ZZ$).
Fold them to obtain tilings $\bt_{\pm 1;i}$ of $\cD \times [0,4]$.
If the tilings $\bt_{\pm 1}$ of $\cD_i \times [0,4]$
are constructed as in Figure \ref{fig:tpm} then
$\bt_{\pm 1;i}$ are essentially
tilings of the corner
$[0,3] \times [0,2]^2 \times [0,3]$ box.
More precisely:
every domino in $\bt_{\pm 1;i}$ is either contained in the box above
or disjoint from it;
the dominoes outside the box
are the same in all four tilings and respect the path.
The tilings $\bt_{\pm 1;i}$ have twist $1 \in \ZZ/(2)$.
Thus, from Fact \ref{fact:2233}, they are all equivalent
(i.e., $\bt_{+1;1} \approx \cdots \approx \bt_{-1;2}$).
Let $\bt_1 = \bt_{+1;1}$, a tiling of $\cD \times [0,4]$:
we have $\bt_1 \ast \bt_1 = e$ (in the domino group).
We claim that $\bt_1$ generates $G^{+}_{\cD}$
(the proof of the claim will complete the proof the theorem).

A domino $d \subset \cD$,
formed by unit cubes $\gamma_0(k_0)$ and $\gamma_0(k_1)$,
{\em can be unfolded} to $\cD_i$ ($i \in \{1,2\}$)
if and only if $\gamma_i(k_0)$ and $\gamma_i(k_1)$ are adjacent.
Notice that a domino in the direction $e_1$ respects the path
and can therefore be unfolded to both $\cD_1$ and $\cD_2$.
A domino in the direction $e_2$ can always be unfolded to $\cD_2$
but usually not to $\cD_1$.
A domino in the direction $e_3$ can always be unfolded to $\cD_1$
but usually not to $\cD_2$.
In particular, every domino that respects the path can be unfolded
to either one among $\cD_1$ and $\cD_2$
and every domino can be unfolded 
to at least one among $\cD_1$ and $\cD_2$

As in the construction detailed in Section \ref{section:hamiltonian},
let $d \subset \cD$ be a domino that does not respect the path $\gamma_0$,
let $\phi$ be a possible value of the flux and
let $\bt_{d;\phi}$ be the corresponding tiling of $\cD\times [0,N]$.
The tiling $\bt_{d;\phi}$ has a unique domino which does not 
respect the path and therefore can be unfolded to obtain
a tiling $\tilde\bt$ of $\cD_i \times [0,N]$
for some choice of $i \in \{1,2\}$.
From Fact \ref{fact:regularrectangle},
there exists a finite sequence of flips taking
$\tilde\bt \ast \bt_{\nvert}$ to a product of finitely many copies
of $\bt_{\pm 1}$ (tilings of $\cD_i \times [0,\tilde N]$ 
for some even $\tilde N \ge N$).
Fold this sequence of flips to obtain a similar sequence
from $\bt_{d;\phi} \ast \bt_{\nvert}$ to a product of finitely many copies
of $\bt_{\pm 1;i}$ (tilings of $\cD \times [0,\tilde N]$).
From what we saw above, 
$\bt_{d;\phi}$ is then equivalent to some power of $\bt_1$.
This completes the proof of the claim and of the theorem.
\end{proof}


\section{Regularity of boxes for $n > 4$}
\label{section:box5}

The following lemma completes the proof of Theorem \ref{theo:regularbox}.

\begin{lemma}
\label{lemma:box5}
Consider $n > 4$.
Consider a cubiculated box $\cD = [0,L_1] \times \cdots \times [0,L_{n-1}]$,
all $L_i \ge 2$, at least one of the $L_i$ even.
The box $\cD$ is regular.
\end{lemma}

\begin{proof}
The proof is by induction on $n$;
Theorem \ref{lemma:box4} serves as the basis of the induction.

Consider $n > 4$ and a box $\cD$ as in the statement.
For $k \le n-2$, let
$\cD_k = [0,L_1] \times \cdots \times
[0,L_kL_{k+1}] \times \cdots \times [0,L_{n-1}]$.
From Example \ref{example:foldbox} we know that each $\cD_k$
can be folded to obtain $\cD$.
By induction, we know that each $\cD_k$ is regular;
notice that if $n = 5$ we still have $L_kL_{k+1} > 2$.
As in the proof of Theorem \ref{lemma:box4},
any domino is compatible with unfolding
to all but possibly one $\cD_k$.
Of course, dominoes which respect the path
can be unfolded to any $\cD_k$.

We first notice that there exist tilings of twist $1 \in \ZZ/(2)$
of $\cD \times [0,N]$ for some even $N$.
As discussed in Section \ref{section:hamiltonian},
any tiling $\bt$ of $\cD \times [0,N]$, $N$ even,
is homotopic to a product of tilings $\bt_{d;\phi}$
containing a single domino which does not respect the path.
Thus, at least one of them has twist $1$:
call it $\bt_1 = \bt_{d_1;\phi_1}$.
Let $\bt_0$ be the vertical tiling of $\cD \times [0,2]$.

We prove that if $\Tw(\bt_{d;\phi}) = 0$ then $\bt_{d;\phi} \sim \bt_0$.
Indeed, $\bt_{d;\phi}$ can be unfolded to some $\cD_k$
to obtain a tiling $\bt_2$ of $\cD_k \times [0,N]$.
By the definition of twist, $\Tw(\bt_2) = 0$.
Since $\cD_k$ is regular, there exists $N_2$ even 
and a sequence of flips in $\cD_k \times [0,N+N_2]$
taking $\bt_2 \ast \bt_{\nvert,N_2}$ to $\bt_{\nvert,N+N_2}$.
Fold this sequence of flips to obtain the desired homotopy in $\cD$.

We prove that if $\Tw(\bt_{d;\phi}) = 1$ then $\bt_{d;\phi} \sim \bt_1$.
Indeed, $d$ rules out at most one value of $k$ (for unfolding)
and $d_1$ rules out at most another value.
There is still at least one value of $k$ such that
$\bt_{d;\phi} \ast \bt_1^{-1}$ can be unfolded to $\cD_k \times [0,N]$.
We have $\Tw(\bt_{d;\phi} \ast \bt_1^{-1}) = 0$ and therefore,
as in the previous paragraph, a sequence of flips in $\cD_k$.
Fold the sequence as above and we are done.
\end{proof}


\section{Proof of Theorem \ref{theo:smallM}}
\label{section:smallM}

We are ready to proceed to the proof of Theorem \ref{theo:smallM}.
The proof is similar to that of Theorem 2 from \cite{regulardisk},
but, due to the finiteness of $G_{\cD}$, significantly simpler.

\begin{proof}[Proof of Theorem \ref{theo:smallM}]
Let $\cD$ be a regular region and $\cC_{\cD}$ be its complex.
Let $\Pi: \tilde\cC_{\cD} \to \cC_{\cD}$ be its universal cover
so that $\tilde\cC_{\cD}$ is a simply connected finite complex.
Let $\tilde\cP$ be the finite set of vertices of  $\tilde\cC_{\cD}$.
Let $\emptyplug \in \tilde\cP$ be a base point,
fixed from now on, satisfying $\Pi(\emptyplug) = \emptyplug\in\cP$.
A tiling of $\cR_{0,N;\emptyplug,p}$ is a walk in $\cC_{\cD}$
and can therefore be lifted to a continuous path in $\tilde\cC_{\cD}$,
starting at $\emptyplug \in\tilde\cP$ and ending in
an element of $\Pi^{-1}[\{p\}] \subset \tilde\cP$.
For each $p \in \tilde\cP$,
let $\bt_p$ be a path in $\tilde\cC_{\cD}$
from $\emptyplug$ to $p$, of length $N_p$.
Assume $\bt_{\emptyplug}$ to be the path of length $0$
so that $N_{\emptyplug} = 0$.

Let $p_0, p_1 \in \tilde\cP$.
For every floor $\bff$ from $p_0$ to $p_1$,
there exists a homotopy fixing endpoints between
$\bt_{p_0} \ast \bff$ and $\bt_{p_1}$.
By construction, there exists an even integer 
$M_{p_0,p_1,\bff} \ge \max\{N_{p_0},N_{p_1}\}$ such that
\[ \bt_{\nvert,M_{p_0,p_1,\bff} - N_{p_0}}
\ast \bt_{p_0} \ast \bff \approx
\bt_{\nvert,M_{p_0,p_1,\bff} - N_{p_1}}
\ast \bt_{p_1}. \]
Let $M$ be even and equal to or larger than
the maximum among all $M_{p_0,p_1,\bff}$.

Consider a tiling $\bt_{\dagger}$ as a (continuous) path of length $N_{\dagger}$
in $\tilde\cC_{\cD}$ from $\emptyplug$ to $p_{\dagger} \in \tilde\cP$.
For $k \in \ZZ$, $0 \le k \le N_{\dagger}$,
let $p_k$ be the $k$-th vertex of the path $\bt_\dagger$
so that $p_0 = \emptyplug$ and $p_{N_{\dagger}} = p_{\dagger}$.
Let $\bt_{\dagger,k}$ be the restriction of the original path $\bt_{\dagger}$
to $[k,N]$ so that
$\bt_{\dagger,k} \in \cT(\cR_{k,N_{\dagger};p_k,p_{\dagger}})$.
We construct a homotopy $H$ from $\bt_{\dagger}$ to $\bt_{p_{\dagger}}$.
For $k \in \ZZ$, $0 \le k \le N_{\dagger}$,
set $H(k) = \bt_{p_k} \ast \bt_{\dagger,k}$.
Notice that $H(0) = \bt_{\dagger}$ and $H(N_{\dagger}) = \bt_{p_{\dagger}}$.
In order to move from $H(k)$ to $H(k+1)$
we proceed as in the previous paragraph,
rewriting $H(k) = \bt_{p_k} \ast \bff \ast \bt_{\dagger,k+1}$.
This step can be accomplished in
$\cR_{M_{p_k,p_{k+1},\bff}+(N_{\dagger}-k-1)}$.
Thus, the entire homotopy can be constructed 
as a sequence of flips in $\cR_{N_{\dagger}+M}$,
completing the proof.
\end{proof}

\begin{coro}
\label{coro:smallMcoro}
Let $\cD \subset \RR^{n-1}$ be a regular region;
let $M$ be as in Theorem \ref{theo:smallM}.
Let $\bt_0, \bt_1$ be tilings of $\cR_N$.
If both $\bt_0$ and $\bt_1$ have at least $M$ vertical floors
and $\Tw(\bt_0) = \Tw(\bt_1)$ then $\bt_0 \approx \bt_1$.
\end{coro}

\begin{proof}
We know that vertical floors can be moved up and down by flips.
In other words, there exist tilings $\bt_{0,\bullet}, \bt_{1,\bullet}$
of $\cR_{N-M}$ with $\bt_i \approx \bt_{i,\bullet} \ast \bt_{\nvert,M}$
(for $i \in \{0,1\}$).
We also have $\Tw(\bt_{i,\bullet}) = \Tw(\bt_i)$ and therefore
$\Tw(\bt_{0,\bullet}) = \Tw(\bt_{1,\bullet})$.
By regularity, $\bt_{0,\bullet} \sim \bt_{1,\bullet}$.
By Theorem \ref{theo:smallM},
$\bt_{0,\bullet} \ast \bt_{\nvert,M} \approx
\bt_{1,\bullet}\ast \bt_{\nvert,M}$, as desired.
\end{proof}

We are now have all the ingredients to prove Corollary \ref{coro:twin}.

\begin{proof}[Proof of Corollary \ref{coro:twin}]
We know that twist partitions $\cT(\cR_N)$ into two subsets
$\tilde T_{+}$ (twist equal to $0 \in \ZZ/(2)$) and
$\tilde T_{-}$ (twist equal to $1$).
We have $|\tilde T_{+}| - |\tilde T_{-}| = \Delta(\cR_N)$.
From Lemma \ref{lemma:lambda}, $|\Delta(\cR_N)|$
is exponentially smaller than $|\cT(\cR_N)|$ 
(as a function of $N$).

Let $M$ be as in Corollary \ref{coro:smallMcoro}.
From Lemma \ref{lemma:verts}, for $N$ sufficiently large,
most tilings of $\cR_N$ admit at least $M$ vertical floors.
Let $\bt_0, \bt_1$ be two such tilings with $\Tw(\bt_i) = i \in \ZZ/(2)$.
Let $T_i \subset \cT(\cR_N)$ be the $\approx$-equivalence class of $\bt_i$.
We have $T_i \subseteq \tilde T_i$.
From Corollary \ref{coro:smallMcoro}, all tilings which have
at least $M$ vertical floors belong to $T_0 \cup T_1$.
From Lemma \ref{lemma:verts},
$|\cT(\cR_N) \smallsetminus (T_0 \cup T_1)|/|\cT(\cR_N)|$
tends to zero exponentially in $N$.
The desired results follow.
\end{proof}

\medskip

\footnotesize

\noindent
Caroline J. Klivans \\
Division of Applied Mathematics, Box F \\
182 George Street \\
Brown University \\
Providence, RI 02912 \\
\url{klivans@brown.edu}

\smallskip

\noindent
Nicolau C. Saldanha \\
Departamento de Matem\'atica, PUC-Rio \\
Rua Marqu\^es de S\~ao Vicente, 225, Rio de Janeiro, RJ 22451-900, Brazil \\
\url{saldanha@puc-rio.br}

\end{document}